\documentclass[draft=false, leqno]{article}
\usepackage{boilerart}
\usepackage{booktabs}
\usepackage{tensor}

\usetikzlibrary{backgrounds,decorations.markings,hobby,fit}

\defbibheading{references}[\refname]{%
	\section*{#1}%
	\addcontentsline{toc}{section}{References} %
	\markboth{#1}{#1}
	}

\DeclareMathOperator{\Arr}{\categ{Arr}}
\DeclareMathOperator{\biCat}{\categ{biCat}}
\DeclareMathOperator{\Bun}{Bun}
\DeclareMathOperator{\Div}{Div}
\DeclareMathOperator{\Env}{\categ{Env}}

\DeclareMathOperator{\Grass}{Grass}

\DeclareMathOperator{\Haus}{\categ{Haus}}
\DeclareMathOperator{\Hilb}{Hilb}
\DeclareMathOperator{\Hecke}{Hecke}
\DeclareMathOperator{\Isol}{\categ{Isol}}

\DeclareMathOperator{\Spa}{Spa}
\DeclareMathOperator{\TwArr}{\categ{TwArr}}

\DeclareMathOperator*{\bigleftsum}{\underline{\bigoplus}}
\DeclareMathOperator*{\bigrightsum}{\overline{\bigoplus}}

\newcommand{\accr}{\textit{accr}}
\newcommand{\disp}{\textit{disp}}

\newcommand{\hop}{\textit{hop}}
\newcommand{\lax}{\textit{lax}}
\newcommand{\leftboxtensor}{\mathbin{\underline{\boxtimes}}}

\newcommand{\leftsum}{\mathbin{\underline{\oplus}}}
\newcommand{\lefttensor}{\mathbin{\underline{\otimes}}}
\newcommand{\MAP}{\underline{\Map}}
\newcommand{\ol}[1]{\overline{#1}}

\newcommand{\ordrightsum}{\mathbin{\overrightarrow{\oplus}}}
\newcommand{\PPi}{\mathbfit{\Pi}}
\newcommand{\para}{\textit{para}}
\newcommand{\partto}{\rightharpoonup}
\newcommand{\paw}{\ensuremath{\textup{Paw}}}

\newcommand{\rightboxtensor}{\mathbin{\overline{\boxtimes}}}

\newcommand{\rightsum}{\mathbin{\overline{\oplus}}}
\newcommand{\righttensor}{\mathbin{\overline{\otimes}}}

\newcommand{\trip}{\textit{trip}}
\newcommand{\triv}{\textit{triv}}
\newcommand{\ul}[1]{\underline{#1}}
\newcommand{\vop}{\textit{vop}}

\title{Factorization algebras in quite a lot of generality}

\author{Clark Barwick}

\date{Version 2 -- February 2026}

\begin{document}

\maketitle

This is a first stab at a mathematical framework in which one can study quantum field theories on spacetimes with quite general geometries.
We will study these theories via their \emph{factorization algebras}.

Factorization algebras capture structures one finds on the spaces of observables of a quantum field theory.
Roughly speaking, if $M$ is a spacetime manifold, then a factorization algebra $F$ on $M$ assigns to any open subset $U \subseteq M$ the space $F(U)$ of measurements that one can perform within $U$.
Given \emph{disjoint} opens $U$ and $V$ and an inclusion $U \sqcup V \subseteq W$, these measurements can be combined, giving rise to a map $F(U) \otimes F(V) \to F(W) $.
Assembling these operations coherently yields the definition of a factorization algebra in the \emph{manifold context}, as studied by Costello--Gwilliam \cite{MR3586504,MR4300181}.

In the \emph{holomorphic context}, factorization algebras capture the formal properties of operator product expansions in holomorphic conformal field theory.
In this story, the spacetime manifold $M$ is replaced by (say) a Riemann surface $X$.
A factorization algebra is then a family of \emph{sheaves} (of some appropriate kind -- often $D$-modules) $A_n$ on the various powers $X^n$, along with data to ensure that if $x_1, \dots, x_n \in X$ are \emph{distinct} points, then one has local identifications
\[
  (A_n)_{(x_1, \dots, x_n)} = (A_1)_{x_1} \otimes \dots \otimes (A_1)_{x_n} \period
\]
The stalks of $A_1$ are the spaces of local operators of the conformal field theory.
The result is the theory of factorization algebras in the sense of Beilinson--Drinfeld (\cite{MR2058353}).

In this formulation, factorization algebras have become a central structure in modern geometric representation theory,
most notably in the monumental conceptual edifice formed around the geometric Langlands conjecture in the two decades since the publication of Beilinson--Drinfeld's book.
See, \eg, the work of 
Francis--Gaitsgory \cite{MR2891861},
Gaitsgory \cite{MR2403306,MR4176538,MR4117995},
Gaitsgory--Lysenko \cite{MR3769731},
Hennion--Kapranov \cite{MR4536901},
Ho \cite{MR3705928,MR3703458},
Henriques \cite{MR3821755},
Raskin \cite{MR3251357,MR4283758}.

The relationship between the manifold context and the holomorphic context appears to be a form of duality.
Costello--Gwilliam suggest \cite{MR3586504} that the costalks of a factorization algebra in the topological sense should agree with the stalks of an attached factorization algebra in the geometric sense.
They prove such a statement over $\AA^1_{\CC}$ by comparison with vertex operator algebras.
In general, it doesn't seem clear \emph{where} such a comparison can actually be made.

Among the factorization algebras on a manifold $M$ are the \emph{locally constant} factorization algebras, which capture the spaces of observables of a topological quantum field theory.
Locally constant factorization algebras on $M$ can equivalently be described as algebras over a suitable operad of embedded disks.
For example, locally constant factorization algebras on $\RR^n$ are $E_n$-algebras.
As this example makes clear, locally constant factorization algebras depend on more than the homotopy type of $M$ (unlike locally constant \emph{sheaves}).
Nonetheless, we would like to think of this as the \emph{homotopical context} for factorization algebras.

To make sense of factorization algebras in these contexts, one typically finds oneself using special features of the geometric objects with which one is working, such as the fact that every point of a manifold is contained in an embedded open disk,
or else one forms large objects like the Ran space, whose geometry is brittle and technical.
However, Raskin identified \cite{MR3251357} -- and Hennion--Kapranov revisited \cite{MR4536901}-- a formalism for factorization structures that largely avoids these particularities and technicalities. 
This formalism is more naturally encoded in the language of factorization algebras in a sheaf-theoretic setup similar to the one in algebro-geometric contexts.

It is tempting to see how far we can push this approach --
to identify a minimalist formalism that makes sense of factorization algebras in any geometric context.
The goal is not generality for its own sake, but to lay bare the skeleton of the theory.

On one hand, this places the theories of factorization algebras in (at least) the holomorphic and homotopical contexts on an equal footing, allowing for meaningful comparisons.
On the other hand, the formalism extends the technology of factorization algebras to many new contexts.
Of particular interest are factorization algebras in various arithmetic contexts.
In particular, the observables of \emph{arithmetic quantum field theories} --
whose existence has been predicted by Kim \cite{MR4166930} and Ben-Zvi--Sakellaridis--Venkatesh \cite{BZSV} --
should form factorization algebras in the context of Clausen--Scholze's analytic geometry.

The purpose of this paper is to give in to this temptation.

\subsection*{Synopsis}
\label{sub:synopsis}

Our first contention: in order to make sense of an algebra of observables on some geometric object $X$,
we must appeal to an additional piece of structure on $X$ -- an \emph{isolability structure} -- which is often left implicit.
An isolability structure on $X$ is the data required to say whether two points of $X$ are \emph{distant} --
or at any rate can be made distant enough to combine measurements performed at those points. 

The need for this structure reflects the principle of \emph{locality} in quantum field theory.
In effect, this principle states that the outcome of an experiment should not depend upon the outcomes of distant experiments.
Haag underscored the importance of this principle in the 1950s, and 
it plays a key role in the Haag--Kastler axiomatics for algebras of observables \cite{MR165864}.
A slightly more refined version of the principle -- the \emph{cluster decomposition principle} -- states that the scattering matrix that describes a pair of distant processes factors into the product of the scattering matrices for each of the individual processes \cite[Ch. 4]{MR2148466}.
Weinberg argued \cite{MR1699412} that this principle, Lorentz invariance, and quantum mechanics together entail quantum field theory.

The question of whether two points of spacetime are distant is generally not a \enquote{yes or no} question, and
it does not presuppose an existing notion of distance on $X$, such as a metric.
Rather, an isolability structure on $X$ provides a \emph{moduli space} of ways for pairs of points to be \emph{isolated}.
We require this moduli space to be a geometric object of the same type as $X$. 
In other words, in order to specify an isolability structure on $X$, one needs to specify, for every pair of $T$-points $x,y \colon T \to X$, a higher groupoid (or \emph{space}) $\enquine{x \nsim y}$.
The space $\enquine{x \nsim y}$ is the space of ways of isolating $x$ and $y$ from each other without altering any observables in any possible quantum field theories.
These spaces coalesce into a single object $X^{1 \oplus 1}$ with a map $X^{1 \oplus 1} \to X \times X$.
On $T$-points, the fiber of $X^{1 \oplus 1}(T) \to X(T) \times X(T)$ over $(x,y)$ is the space $\enquine{x \nsim y}$.
Thus $X^{1 \oplus 1}$ is a configuration space of pairs of isolated points in $X$.

The theory of isolability objects is complementary to the theory of stacks.
Between two $T$-points $x,y$ of a stack $Y$ one has the path-space space $\enquine{x = y}$, which is the space of ways for $x$ and $y$ to be the \enquote{same} --
or at any rate the space of routes by which we may connect $x$ and $y$.
If $Y_1$ is the path space of $Y$, then the space $\enquine{x = y}$ is the fiber of the endpoint map $Y_1(T) \to Y(T) \times Y(T)$ over $(x,y)$.
A stack is in effect a diagram $Y_{\bullet}$ of higher path spaces and maps between them, indexed by a combinatorial category such as Kan's simplex category $\Delta$.
Analogously, we will define an isolability object $X^{\bullet}$ as a diagram of configuration spaces $X^{\lambda}$ and maps between them.
For instance, the structure of an isolability object will include $X^{2 \oplus 3 \oplus 2}$, which is the space of septuples $(x_1, \dots, x_7)$ of points of $X$ organized in three \enquote{clusters} $\{x_1,x_2\}$, $\{x_3,x_4,x_5\}$, and $\{x_6,x_7\}$.
This diagram is indexed by a certain category $\DD$ of graphs called \emph{cographs},
which we study in detail in \cref{sec:combinatorics}.
The formal properties of this category of cographs and its sibling $\EE$ turn out to be central to our narrative --
on one hand to study isolability structures in \cref{sec:isolability}, and
on the other hand to study \emph{twofold symmetric monoidal structures}, which we will discuss later in this introduction and in detail in \cref{sec:twofold}.

Isolability structures have in many examples been left implicit because one already has access to a set-theoretic description of $X$,
with respect to which one can meaningfully ask whether two points of $X$ are \emph{literally distinct}.
For example, if $X$ is a manifold, then we typically let $X^{1 \oplus 1}$ be the complement of the diagonal in $X \times X$.

In contrast, suppose we want to consider \emph{locally constant} factorization algebras on $X$ --
the observables for topological quantum field theories.
Since locally constant sheaves depend only on the homotopy type $\Pi(X)$, we might try to start there, but
we have already seen that we cannot hope to define locally constant factorization algebras on $X$ in a way that only makes reference to $\Pi(X)$.
So what additional structure is needed?
Once we include the information of the various (stratified) homotopy types $\Pi(X^{\lambda})$ and the maps between them,
we \emph{can} recover locally constant factorization algebras on $X$.
This is precisely the data of an \emph{isolability space}, which we regard as an enhancement of the homotopy type of $X$.
Isolability spaces are the basic objects in the \enquote{homotopical context}, which we explore in \cref{sec:homotopical}.

For example, in the isolability space attached to $\RR^n$, we have a point $x$, which is unique up to a contractible choice, but the space $\enquine{x \nsim x}$ is an $(n-1)$-sphere.
This is notationally oxymoronic, but
here it's helpful to think of $\enquine{x \nsim x}$ as the space of ways to move two copies of $x$ away from one another without altering the observables of any topological quantum field theory.
This is exactly the feature that allows us to relate locally constant factorization algebras on $\RR^n$ to $E_n$-algebras.
Indeed, factorization algebras $F$ only allow measurements to be combined when they happen on isolated pairs of points, so we obtain a family of multiplications $F_x \otimes F_x \to F_x$ parametrized by $S^{n-1}$, precisely as the structure of an $E_n$-algebra demands.

In fact, it turns out that the isolability space attached to $\RR^n$ can be identified with a combinatorially-defined object.
This we show also in \cref{sec:homotopical}.
We relate this to Cepek's combinatorial description of the Ran space of $\RR^n$,
from which the identification of locally constant factorization algebras on $\RR^n$ with $E_n$-algebras is an immediate consequence.

Speaking of the Ran space, our story so far has not made it clear what advantage, if any, an isolability structure might have over the usual Ran space.
After all, the purpose of the Ran space is to parametrize tuples of distinct points of a geometric object.
Indeed, for many examples (the \enquote{$2$-skeletal} isolability spaces), giving an isolability structure is equivalent to giving the Ran space as a commutative monoid in spans.

However, in many examples of quantum field theories, one wants to keep track of not only observables at points, but along \emph{extended objects}.
So in many cases, it is not actually the \enquote{spacetime object} $X$ itself that needs to possess the isolability structure, but rather a related object we call the \emph{observer stack} $O_X$.
The observer stack is an additional datum, which can be extremely general;
in effect it is some moduli space of maps (often embeddings) into $X$, which expresses the options for the support an observable may have.
For example, if $X$ is a variety over the complex numbers, then one may define $O_X$ as a Hilbert stack of subvarieties of $X$.
We discuss this in \cref{sub:observer}.

In this sort of setting, we do not know a definition of Ran space that is sufficient to recover this.
Hennion, Melani, and Vezzosi \cite{hennionmelanivezzosi} have taken a step in this direction for a surface over $\CC$, and
their picture of factorization structures allows for more sophisticated combinations of observables than what we manage to do here.
(See, however, \ref{sub:comments-factorization} for a speculative strategy to capture such combinations.)

\bigskip

\noindent Our second contention: the geometry of isolability objects is controlled by \emph{twofold symmetric monoidal structures}, and factorization structures are defined in terms of these.
To explain this slogan, we have to say what we mean both by \enquote{the geometry} of isolability objects and by \enquote{twofold} symmetric monoidal structures.

In the middle of the 20th century, the French algebraic geometry school realized that the structure one needs on a topological space (or similar object) $X$ in order to \enquote{do geometry} is a sheaf $\mathscr{O}$ of local rings (or similar objects).
The sections of $\mathscr{O}$ are \emph{functions}.
We require that $(X,\mathscr{O})$ is locally isomorphic to some basic objects
(\textit{e.g.}, Euclidean spaces with smooth functions, or complex Stein manifolds with holomorphic functions, or affine schemes with regular functions, or affinoid adic spaces with their functions).
Sheaves of rings give rise to sheaves of module categories: 
one starts with suitable categories $\AA(X)$ of $\mathscr{O}$-modules on the basic objects and extends $\AA$ to all of our geometric objects by gluing. 
For example, this is how we extend the assignment $\Spec R \mapsto \Mod(R)$ to the assignment $S \mapsto \QCoh(S)$ on schemes. 

Over the course of the development of algebraic geometry in the 1960s, it became clear that the sheaf of symmetric monoidal categories $X \mapsto \AA(X)$ is really what contains the geometric content.
For example, $\Ext$-groups in $\AA(X)$ are cohomology groups of $X$, and 
all the various features of cohomology -- structures, functorialities, and dualities -- came to be elegantly expressed in terms of the sheaf $\AA$.
It no longer particularly mattered whether $\AA(X)$ happens to be a category of modules over its unit object.
One can, for example, look at varieties over a finite field through the lens of the cohomology of $\el$-adic sheaves, and tell a very different story about their geometric behavior than what one gets from the cohomology of quasicoherent sheaves.

It is thus natural to categorify our notion of geometric object:
instead of a ringed space $(X,\mathscr{O})$, we consider a pair $(\XX,\AA)$ consisting of a category $\XX$ of geometric objects along with a sheaf of symmetric monoidal categories $\AA$.%
\footnote{
  The full apparatus of geometry becomes available when $(\XX,\AA)$ is part of a \emph{six functor formalism}.
  We hope to return to the interaction of this structure with the theory of factorization algebras in future work.
}
We think of $\AA$ as the \enquote{theory of sheaves} on the objects of $\XX$.

An isolability structure on an object $X \in \XX$ now creates a kind of \enquote{parallax} on the symmetric monoidal category $\AA(X)$ of sheaves.
Indeed, consider what happens if we simply apply our functor $\AA$ to each object $X^{\lambda}$.
The result is a diagram of categories $\AA(X^{\bullet})$ indexed by our combinatorial category $\DD$.
This diagram is multiplicative in \emph{two} senses.
To illustrate, suppose that $F,G \in \AA(X)$.
The \enquote{external} tensor product $F \boxtimes G$ lies in $\AA(X^2)$ whose stalk at $(x,y)$ is $F_x \otimes G_y$.
The two multiplicative structures on $\AA(X^{\bullet})$ arise from restricting $F \boxtimes G$ to the locus $\{x = y\}$ on one hand, or to the locus $\{x \neq y\}$ on the other.
That is, one obtains the usual tensor product -- which we write as $F \righttensor G$ -- by pulling back $F \boxtimes G$ along the diagonal $X^1 \inclusion X^2$;
dually, one may also pull back $F \boxtimes G$ along the \enquote{complement} $X^{1 \oplus 1} \inclusion X^2$ to get a different tensor product, which we write as $F \lefttensor G$.
The interplay between these two multiplicative structures is precisely the algebraic feature that allows the theory of factorization algebras.

Our combinatorial categories offer a comfortable vantage point from which to study this structure.
The category $\DD$ has \emph{two} symmetric monoidal structures:
a \emph{left} symmetric monoidal structure $\leftsum$, which is the disjoint union of cographs, and
a \emph{right} symmetric monoidal structure $\rightsum$, which is the \enquote{connected sum} or \enquote{join} of cographs.
The right and left structures do not intertwine in a strict sense
(for otherwise they would coincide by Eckmann--Hilton!).
Rather, there is a noninvertible \emph{intertwiner map}
\[
  (A \rightsum B) \leftsum (C \rightsum D) \to (A \leftsum C) \rightsum (B \leftsum D) \period
\]
Category theorists have had language for this structure for some time \cite{MR1975810,MR1982884}:
$\DD$ is a \emph{twofold symmetric monoidal category}.
(There is also the slightly more general notion of \emph{duoidal bicategory} and \emph{duoid} therein \cite{MR2724388,MR3040601}.)
Earlier still, Borcherds \cite{MR1865087} highlighted the importance of twofold monoidal categories (without using the name) already in the late 1990s.
At around the same time, Beilinson and Drinfeld made use of a variant of this structure in their big text \cite{MR2058353} on chiral algebras, where they are called \enquote{compound tensor categories}.
With our categories of cographs in hand, it is not difficult to set up a working theory of twofold symmetric monoidal categories in general (\cref{sec:twofold}).

In fact our category $\DD$ is the \emph{universal} twofold symmetric monoidal category containing a commutative monoid for the left tensor product.
That means that a symmetric monoidal category can be uneconomically described as a twofold symmetric monoidal functor $\DD \to \Cat$.
This lack of economy allows us to introduce what we call \emph{parallax}.
A \emph{parallax symmetric monoidal category} is a \emph{lax} symmetric monoidal functor $\DD \to \Cat$
(or, equivalently, a functor from $\DD$ to symmetric monoidal categories).
In particular, $\AA(X^{\bullet})$ is a parallax symmetric monoidal category.

Now if $\VV_{\bullet}$ and $\WW_{\bullet}$ are parallax symmetric monoidal categories, then
we can ask for a functor $\VV_{\bullet} \to \WW_{\bullet}$ to be either \emph{left} or \emph{right} symmetric monoidal.
\emph{Factorization algebras} on $X$ with coefficients in the sheaf theory $\AA$ are functors $1_{\bullet} \to \AA(X^{\bullet})$ that are right symmetric monoidal. 

When $X^{\bullet}$ is an isolability space and $\AA$ is the theory of constructible sheaves,
we obtain (\ref{sub:constructible}) the theory of \emph{constructible} and \emph{locally constant}
factorization algebras.
These are the factorization algebras that describe topological quantum field theories.

Factorization stacks (\ref{sub:factorizationstacks}) are the most primitive and unstructured kind of factorization algebra.
That's not to say that there aren't interesting examples, however.
We give a simple example -- the \emph{Beilinson--Drinfeld Grassmannian}, in quite a lot of generality.
For an object $X$ and a pointed stack $B$ on $X$, we construct (\ref{sub:heckegrass}) the Grassmannian factorization stack on a suitable observer stack $O_X^\bullet$.
When $X$ is a curve over $\CC$, $B = BG$ for a group scheme $G$ over $X$, and $O_X = X$, this Grassmannian recovers the one introduced by Beilinson--Drinfeld.
With our formalism, we may take $X$ to be the Fargues--Fontaine curve and $O_X$ the \enquote{mirror curve} $\Div^1$;
this amounts to a factorization enhancement of a construction of Scholze--Weinstein \cite{MR4446467}.

\subsection*{Limitations}
\label{sub:Limitations}
We paint a picture of factorization structures that works in quite a lot of generality.
But only quite.
Our approach still suffers from a number of defeciencies,
of which we highlight only four.
\begin{enumerate}
  \item In this paper, we do not incorporate the \emph{manifold} context for factorization algebras --
    \ie, factorization algebras in the sense of Costello--Gwilliam --
    into our setup.
    We aim to return to this point in later work.
  \item It is not clear to us how to develop what might be called a \enquote{$\beta$-version} of factorization algebras in our picture.
    This would provide structures for combining extended observables/operators in more complicated ways than we manage to do here.
    We offer some speculations about the algebraic format of such a theory at the end of the paper (\ref{sub:comments-factorization}).
    At this point, there is much we do not know. 
  \item In this paper we do not attempt to look through the Koszul dual lens of the Lie-algebra description of chiral algebras, and
    we do not currently fully understand how to incorporate that story into our picture.
    This seems a promising avenue of study.
  \item Perhaps the least satisfying aspect of our picture is that it is only a picture.
    Our eventual aim is to make sense of arithmetic quantum field theories, but
    this paper only identifies some first pieces of the formalism.
    The objects of study here do seem to be \enquote{correct} in the weak sense that
    they capture the structures observed in certain flavors of quantum field theory, and
    they are at the same time the product of natural principles and a simple formalism. 
\end{enumerate}

\subsection*{Conventions}
\label{sub:Conventions}
This paper is written with \enquote{implicit $\infty$} conventions, so all algebraic structures (categories, topoi, operads, commutative algebras, \etc) are to be interpreted in the proper homotopical sense ($\infty$-categories, $\infty$-topoi, $\infty$-operads, $E_{\infty}$-algebras, \etc) by default.

We use the term \enquote{space} for what might elsewhere be called \enquote{anima} or \enquote{$\infty$-groupoid} or \enquote{homotopy type}.

All topological spaces should be taken to be compactly generated by default, and all constructions performed with topological spaces (products, mapping spaces, passage to subspaces, \etc) should be k-ified as needed.

\subsection*{Thanks}%
\label{sub:thanks}
The long development period for this work means that I now have the joy of thanking a lot of people.
The ideas here slowly began to coagulate in 2018 in a conversation with Tomer Schlank.
I'm grateful to him for listening to my clumsy early efforts to make sense of arithmetic quantum field theories.

Since then, Willow Bevington, Malthe Sporring, and Lucy Spouncer have evinced even more patience; 
they have tolerated dozens of versions of this story as it's evolved.
Their comments have been instrumental in helping me understand the unfamiliar combinatorial structures that bespeckle this paper.
Malthe in particular has helped me iron out the definition of twofold symmetric monoidal structures, and
together we have really begun to understand the strange new world of \emph{twofold algebra}.

I learned a lot from the participants in the GRIFT seminar at the Maxwell Institute.
Conversations with Tudor Dimofte, Jan Pulmann, and Dominik Rist have been particularly valuable and inspiring.

Behrang Noohi and Magnus Carlson made me aware of the connection to Borcherds' work on quantum vertex algebras \cite{MR1865087}.

Insights from Owen Gwilliam and Pavel Safronov are everywhere in this paper.
Owen kindly read a draft of this paper and offered many suggestions for improvement and words of encouragement. 

Emily Roff has helped me articulate a number of aspects of this story.
In one particular conversation, she insisted that I take the relative free constructions from onefold to twofold symmetric monoidal categories more seriously.
That ended up being vital to my understanding of the real role of twofold algebra in this story.

Special thanks to Christine Barwick, Ron Barwick, Scott Morgan (\textsc{aka} Loscil), Alex Sear, Paul Simeone, Lou Vaickus, and Maggie.

\bookmark[page=1,level=1]{Contents}
\setcounter{tocdepth}{2}
\tableofcontents

\section{The combinatorics of isolation}%
\label{sec:combinatorics}

The combinatorial structures that underlie all our work in this paper are special graphs called \emph{cographs} \cite{MR619603}.
On one hand, the vertices of cographs index tuples of points in a space where certain pairs --
those connected by an edge in the graph --
are required to be distinct.
This leads us to the theory of \emph{isolability objects} (\cref{sec:isolability}).
On the other hand, cographs also index multiplication laws in categories with two tensor products that intertwine in a lax sense.
These are \emph{twofold monoidal categories} (\cref{sec:twofold}).
The interactions between these structures is what gives rise to the theory of factorization algebras.

Fortunately, cographs are so simple that we can classify them completely.
Our approach is to exploit that fact and to develop a fine understanding of these combinatorics.

\subsection{Cographs}%
\label{sub:cographs}
For us, a \emph{graph} $\angs{\lambda} = (V,E)$ consists of a set $V = V\angs{\lambda}$ of \emph{vertices} along with a symmetric relation $E = E\angs{\lambda} \subseteq V \times V$ of \emph{edges}.
We do not assume that $E$ is irreflexive;
in other words we permit our graphs to have loops.
If we want to dispense with them, we may pass to the maximal irreflexive subgraph $\angs{\lambda}_{\textit{irr}} \subseteq \angs{\lambda}$.
If we want them all, we may pass to the reflexive hull $\angs{\lambda}_{\textit{refl}} \supseteq \angs{\lambda}$.

We are interested in graphs $\angs{\lambda} = (V,E)$ that satisfy the following condition:
\begin{equation}\label{eqn:P4freeness}
  \forall w,x,y,z \in V \quad \{(w,x), (y,x), (y,z)\} \subseteq E \implies \{(y,w), (w,z), (z,x)\} \cap E \neq \varnothing \period
\end{equation}
For reasons we will explain, such graphs
(or, more usually, their maximal irreflexive subgraphs)
are called \enquote{$P_4$-free graphs} or \enquote{complement-reducible graphs};
graph-theorists have shortened this latter phrase to the word \enquote{cograph}.
The appeal of this term's brevity is greater than our discomfort with its uncategorical adoption of the preposition \enquote{co}.
So we call every graph $\Gamma$ satisfying the above condition a \emph{cograph}.

The \emph{trivial} cograph structure on a set $V$ is $V_{\triv}=(V,\varnothing)$,
in which no two elements are isolated from each other.
For every $n \geq 0$, we define:
\[
  \angs{\ul{n}} \coloneq \{1,\dots,n\}_{\triv} \period
\]

If $\angs{\lambda} = (V,E)$ is a cograph, then its \emph{negation} $\neg \angs{\lambda} \coloneq (V, (V \times V) \smallsetminus E)$ is a cograph as well.
So, at the opposite extreme from trivial cographs, we have the \emph{complete cographs} $(V, V \times V)$
in which all pairs of elements are isolated from each other.
For every $n \geq 0$, we define:
\[
  \angs{\ol{n}} \coloneq \neg \angs{\ul{n}} \period
\]

Here, for example, is $\angs{\ol{5}}$:
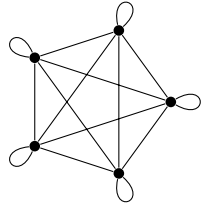
\begin{center}
\begin{tikzpicture}
  \tikzstyle{vertex}=[circle,fill=black,minimum size=4pt,inner sep=0pt]
  \node[vertex] (v_1) at (1,0) {};
  \node[vertex] (v_2) at (0.309,0.951) {};
  \node[vertex] (v_3) at (-0.809,0.588)  {};
  \node[vertex] (v_4) at (-0.809,-0.588)  {};
  \node[vertex] (v_5) at (0.309,-0.951)  {};
  \draw (v_1) -- (v_2) -- (v_3) -- (v_4) -- (v_5) -- (v_1) -- cycle;
  \draw (v_1) -- (v_3) -- (v_5) -- (v_2) -- (v_4) -- (v_1) -- cycle;
  \draw (v_1) to[in=30,out=-30,loop] (v_1);
  \draw (v_2) to[in=102,out=42,loop] (v_2);
  \draw (v_3) to[in=174,out=114,loop] (v_3);
  \draw (v_4) to[in=-174,out=-114,loop] (v_4);
  \draw (v_5) to[in=-102,out=-42,loop] (v_5);
\end{tikzpicture}
\end{center}
The graph $K_n$ is the maximal irreflexive subgraph $\angs{\ol{n}}_{\textit{irr}}$. 

A nonexample is the graph $P_4$: 
\begin{tikzpicture}
  \tikzstyle{vertex}=[circle,fill=black,minimum size=4pt,inner sep=0pt]
  \node[vertex] (v_1) at (-1.5,0) {};
  \node[vertex] (v_2) at (-0.5,0) {};
  \node[vertex] (v_3) at (0.5,0)  {};
  \node[vertex] (v_4) at (1.5,0)  {};
  \draw (v_1) -- (v_2) -- (v_3) -- (v_4) -- cycle;
\end{tikzpicture}.
In fact, this is the \emph{universal} nonexample.
Indeed, complete reducibility is a \emph{hereditary} property;
that is, if $\angs{\lambda} = (V,E)$ is a cograph, then so is any induced subgraph $\angs{\mu} = (W,E|W) \subseteq \angs{\lambda}$.
It follows that a graph $\angs{\lambda}$ is a cograph iff $\angs{\lambda}_{\textit{irr}}$ does not contain the graph $P_4$ as an induced subgraph
(whence the alternative name for cographs -- \enquote{$P_4$-free graphs}).

Every equivalence relation is a cograph.
A cograph $\angs{\lambda}$ is an \emph{apartness relation} iff $\neg\angs{\lambda}$ is an equivalence relation.
Apartness relations are also called \enquote{inequivalence} relations.
For example, consider the complete bipartite graph 
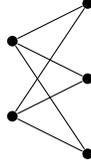
\begin{center}
\begin{tikzpicture}
  \tikzstyle{vertex}=[circle,fill=black,minimum size=4pt,inner sep=0pt]
  \node[vertex] (v_1) at (-0.5,-0.5) {};
  \node[vertex] (v_2) at (-0.5,0.5) {};
  \node[vertex] (v_3) at (0.5,-1)  {};
  \node[vertex] (v_4) at (0.5,0)  {};
  \node[vertex] (v_5) at (0.5,1)  {};
  \draw (v_1) -- (v_3) -- (v_2) -- (v_4) -- (v_1) -- cycle;
  \draw (v_1) -- (v_5) -- (v_2) -- cycle;
\end{tikzpicture}
\end{center}
This is an apartness relation, normally denoted $K_{2,3}$.
In the notation we discuss below (\cref{sub:sums}), this is $\angs{\ul{2} \rightsum \ul{3}}$.

More generally, if $f \colon V \to W$ is a map of sets, then
we obtain an apartness relation $V_{\!f}$ with vertex-set $V$ in which $x$ is isolated from $y$ iff $f(x) \neq f(y)$.

In this paper, the irreflexive cographs are used to make sense of \emph{isolability structures} (\cref{sec:isolability}),
which are structures on a geometric object that permit one to say whether two points are \enquote{different} or \enquote{distinct}.
The reflexive cographs are used later to make sense of \emph{twofold symmetric monoidal structures} (\cref{sec:twofold}),
which are pairs of symmetric monoidal structures related by an intertwiner map.

\subsection{Categories of isolation}%
\label{sub:categories}
A proliferation of indexing categories will appear in this paper.
The most important for us are the following six $1$-categories:
\begin{itemize}
  \item[$\FF$,] the category of finite sets;
  \item[$\GG$,] the category of finite cographs and relation-preserving maps;
  \item[$\DD$,] the category of finite irreflexive cographs and relation-preserving maps;
  \item[$\EE$,] the category of finite reflexive cographs and relation-preserving maps;
  \item[$\DD_{\leq 2}$,] the category of finite apartness relations and isolation-preserving maps;
  \item[$\EE_{\leq 2}$,] the category of finite equivalence relations and equivalence-preserving maps.
\end{itemize}
Mnemonic: $\FF$ for \enquote{finite},
$\GG$ for \enquote{graph},
$\DD$ for \enquote{different} or \enquote{distinct} or \enquote{disjoint},
$\EE$ for \enquote{equivalent}.
The meaning of the number $2$ will become more apparent in \cref{sub:depthfiltration}.

We have a diagram
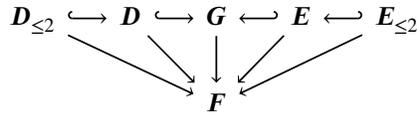
\[
  \begin{tikzcd}[sep=1.5em, ampersand replacement=\&]
    \DD_{\leq 2} \arrow[r,hook] \arrow[drr] \& \DD \arrow[r,hook] \arrow[dr] \& \GG \arrow[d] \& \EE \arrow[l,hook'] \arrow[dl] \& \arrow[l,hook'] \arrow[dll] \EE_{\leq 2} \\
    {} \& {} \& \FF \& {} \& {} 
  \end{tikzcd}
\]
The forgetful functors to $\FF$ carry a graph $\angs{\lambda}$ to its set $V\angs{\lambda}$ of vertices.
These are all cartesian fibrations;
we will refer to the five categories on the top row as \emph{fibered categories} over $\FF$.
The horizontal inclusions are fully faithful, and they preserve and reflect the cartesian maps.

We may define nonunital variants of these categories by letting $\FF_s \subset \FF$ be the subcategory of nonempty finite sets and surjections, and
then pulling back the diagram above along the inclusion $\FF_s \inclusion \FF$ to define $\GG_s$, $\DD_s$, $\EE_s$, $\DD_{s,\leq 2}$, and $\EE_{s,\leq 2}$.

The negation operation $\Gamma \mapsto \neg\Gamma$ is not a functor on $\GG$.
It is however a functor on the wide subcategory $\iota \GG \subset \GG$ of isomorphisms, and 
even better, it's a functor in the \emph{fibered opposite} of $\GG$ over $\FF$, as we shall see (\cref{sub:duality}).
In any case, $\neg$ exchanges the objects of $\DD_{\leq 2}$ and $\EE_{\leq 2}$, and it exchanges the objects of $\DD$ and $\EE$.

Let us write $\angs{n} \coloneq \{1,\dots,n\} \in \FF$.
The assignment $\angs{n} \mapsto \angs{\underline{n}}$ is a fully faithful left adjoint cartesian section into the fibered categories $\DD_{\leq 2}$, $\DD$, and $\GG$. 
Dually, the assignment $\angs{n} \mapsto \angs{\ol{n}}$ is a fully faithful right adjoint cartesian section into the fibered categories $\EE_{\leq 2}$, $\EE$, and $\GG$.

\begin{wrnnn}
  The inclusions $\DD_{\leq 2} \subset \GG$ and $\DD \subset \GG$ preserve all the colimits that exist in $\DD_{\leq 2}$ and $\DD$, but
  these subcategories are not stable under finite colimits.
  For example, let us say that
  maps $\phi \colon \angs{\lambda'} \to \angs{\lambda}$ and
  $f \colon \angs{\lambda'} \to \angs{\mu'}$ in $\DD$
  are \emph{incompatible} iff
  there exist $i,j \in V\angs{\lambda'}$ such that
  $(\phi(i),\phi(j)) \in E\angs{\lambda}$, but
  $f(i) = f(j)$ in $\angs{\mu'}$.
  In this case, the span that $\phi$ and $f$ form cannot be completed to a square in $\DD$.
  Its pushout exists in $\GG$, however. 
  The simplest example of this phenomenon: let $f$ be the identity as a map $\angs{\underline{2}} \to \angs{\ol{2}}_{\textit{irr}}$, and
  let $\phi$ be the unique map $\angs{\underline{2}} \to \angs{\underline{1}}$;
  these are incompatible, but the pushout in $\GG$ is $\angs{\ol{1}}$.
\end{wrnnn}

\subsection{Sums of cographs}%
\label{sub:sums}
If $\angs{\lambda} = (V,E)$ and $\angs{\mu} = (W,F)$ are cographs, then
let $\angs{\lambda \rightsum \mu} = \angs{\lambda} \rightsum \angs{\mu} \coloneq (V \sqcup W, R)$, where $R$ is the \emph{largest} relation such that $E = R|\angs{\lambda}$ and $F = R|\angs{\mu}$ -- \emph{i.e.}, such that $\angs{\lambda}$ and $\angs{\mu}$ are induced subgraphs. 
We call this the \emph{connected sum} of $\angs{\lambda}$ and $\angs{\mu}$.
Dually, let $\angs{\lambda \leftsum \mu} = \angs{\lambda} \leftsum \angs{\mu} \coloneq (V \sqcup W,S)$, where $S$ is the \emph{smallest} relation such that $\angs{\lambda}$ and $\angs{\mu}$ are induced subgraphs. 
We call this the \emph{disconnected sum} of $\angs{\lambda}$ and $\angs{\mu}$.

Both the connected sum $\angs{\lambda \rightsum \mu}$ and the disconnected sum $\angs{\lambda \leftsum \mu}$ are cographs.
In fact, these operations define two distinct symmetric monoidal structures on $\GG$, with common unit $\varnothing$.
Each of these symmetric monoidal structures restricts to symmetric monoidal structures on $\DD$ and $\EE$.
The connected sum $\rightsum$ restricts to a symmetric monoidal structure on $\DD_{\leq 2}$, and
the disconnected sum $\leftsum$ restricts to a symmetric monoidal structure on $\EE_{\leq 2}$.

The sums $\rightsum$ and $\leftsum$ are dual with respect to negation: $\neg\angs{\lambda \rightsum \mu} = \neg \angs{\lambda} \leftsum \neg \angs{\mu}$.
When we reflect on this, two natural transformations become evident.

First, the identity map on the disjoint union $V\angs{\lambda} \sqcup V\angs{\mu}$ can be viewed as a morphism $\angs{\lambda \leftsum \mu} \to \angs{\lambda \rightsum \mu}$.
More subtly, the natural bijection
\[
  V\angs{\lambda} \sqcup V\angs{\mu} \sqcup V\angs{\nu} \sqcup V\angs{\xi} \equivalence V\angs{\lambda} \sqcup V\angs{\nu} \sqcup V\angs{\mu} \sqcup V\angs{\xi}
\]
defines a natural morphism
\[
  \angs{(\lambda \rightsum \mu) \leftsum (\nu \rightsum \xi)} \to \angs{(\lambda \leftsum \nu) \rightsum (\mu \leftsum \xi)} \comma
\]
which we call the \emph{intertwiner}.

This intertwiner is not an isomorphism,
for otherwise Eckmann--Hilton would imply that $\leftsum = \rightsum$.
Rather, the intertwiner exhibits $\leftsum$ as a normal oplax symmetric monoidal functor with respect to $\rightsum$
or, equivalently, $\rightsum$ as a normal lax symmetric monoidal functor with respect to $\leftsum$.
The data $(\GG,\leftsum,\rightsum)$ comprise a \emph{twofold symmetric monoidal category} in the sense of Balteanu, Fiedorowicz, Schwänzl, and Vogt \cite{MR1982884}.
We will describe a homotopy-coherent version of this structure below (\cref{sec:twofold}).

The category $\FF$ is symmetric monoidal with respect to disjoint union $\sqcup$.
The forgetful functor $\GG \to \FF$ is symmetric monoidal relative to both of the symmetric monoidal structures on $\GG$.
That is, $(\GG,\leftsum,\rightsum) \to (\FF,\sqcup,\sqcup)$ is a twofold symmetric monoidal functor.

The sums provide us with a pleasant notation for cographs.
We have already introduced the trivial cograph $\angs{\ul{n}}$ and the complete cograph $\angs{\ol{n}}$.
We can now combine these using our two sums.

For example, we denote by $\angs{\ol{2} \leftsum \ol{2} \leftsum \ol{1}}$ the equivalence relation on $\{1,2,3,4,5\}$ in which the equivalence classes are $\{1,2\}$, $\{3,4\}$, and $\{5\}$.
Its negation is the apartness relation $\angs{\ul{2} \rightsum \ul{2} \rightsum \ul{1}}$:
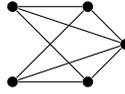
\[
\begin{tikzpicture}
  \tikzstyle{vertex}=[circle,fill=black,minimum size=4pt,inner sep=0pt]
  \node[vertex] (a_1) at (-1,-0.5) {};
  \node[vertex] (a_2) at (-1,0.5) {};
  \node[vertex] (b_1) at (0,-0.5)  {};
  \node[vertex] (b_2) at (0,0.5)  {};
  \node[vertex] (c_1) at (0.5,0)  {};
  \draw (a_1) -- (b_1) -- (a_2) -- (b_2) -- (a_1) -- cycle;
  \draw (a_1) -- (c_1) -- (a_2) -- cycle;
  \draw (b_1) -- (c_1) -- (b_2) -- cycle;
\end{tikzpicture}
\]

Here's another example.
The intertwiner $\angs{(\ul{1} \rightsum \ul{2})\leftsum(\ul{1} \rightsum \ul{2})} \to \angs{(\ul{1} \leftsum \ul{1}) \rightsum (\ul{2} \leftsum \ul{2})} = \angs{\ul{2} \rightsum \ul{4}}$ is the inclusion of the disconnected graph
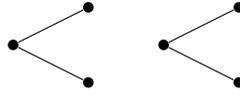
\[
\begin{tikzpicture}
  \tikzstyle{vertex}=[circle,fill=black,minimum size=4pt,inner sep=0pt]
  \node[vertex] (a) at (-1.5,0) {};
  \node[vertex] (b_1) at (-0.5,0.5) {};
  \node[vertex] (b_2) at (-0.5,-0.5)  {};
  \node[vertex] (c) at (0.5,0)  {};
  \node[vertex] (d_1) at (1.5,0.5)  {};
  \node[vertex] (d_2) at (1.5,-0.5)  {};
  \draw (b_1) -- (a) -- (b_2) -- cycle;
  \draw (d_1) -- (c) -- (d_2) -- cycle;
\end{tikzpicture}
\]
into the complete bipartite graph
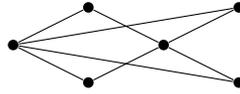
\[
\begin{tikzpicture}
  \tikzstyle{vertex}=[circle,fill=black,minimum size=4pt,inner sep=0pt]
  \node[vertex] (a) at (-1.5,0) {};
  \node[vertex] (b_1) at (-0.5,0.5) {};
  \node[vertex] (b_2) at (-0.5,-0.5)  {};
  \node[vertex] (c) at (0.5,0)  {};
  \node[vertex] (d_1) at (1.5,0.5)  {};
  \node[vertex] (d_2) at (1.5,-0.5)  {};
  \draw (a) -- (b_1) -- (c) -- (b_2) -- (a) -- (d_1) -- (c) -- (d_2) -- (a) -- cycle;
\end{tikzpicture}
\]

By nesting more sums, we can notate increasingly complicated cographs, such as
\[
  \angs{(\ul{2} \rightsum \ul{1}) \leftsum (\ul{4} \rightsum (\ol{2} \leftsum \ol{2}))} \comma
\]
which is
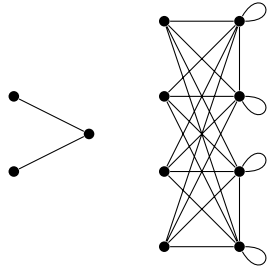
\[
\begin{tikzpicture}
  \tikzstyle{vertex}=[circle,fill=black,minimum size=4pt,inner sep=0pt]
  \node[vertex] (a_1) at (1.5,-1.5) {};
  \node[vertex] (a_2) at (1.5,-0.5) {};
  \node[vertex] (b_1) at (1.5,0.5)  {};
  \node[vertex] (b_2) at (1.5,1.5)  {};
  \node[vertex] (c_1) at (0.5,-1.5)  {};
  \node[vertex] (c_2) at (0.5,-0.5)  {};
  \node[vertex] (c_3) at (0.5,0.5)  {};
  \node[vertex] (c_4) at (0.5,1.5)  {};
  \node[vertex] (d_1) at (-0.5,0)  {};
  \node[vertex] (e_1) at (-1.5,-0.5)  {};
  \node[vertex] (e_2) at (-1.5,0.5)  {};
  \draw (a_1) to[in=0,out=-60,loop] (a_1);
  \draw (a_2) to[in=0,out=60,loop] (a_2);
  \draw (a_1) -- (a_2) -- cycle;
  \draw (b_1) to[in=0,out=-60,loop] (b_1);
  \draw (b_2) to[in=0,out=60,loop] (b_2);
  \draw (b_1) -- (b_2) -- cycle;
  \draw (a_1) -- (c_1) -- (a_2) -- (c_2) -- (b_1) -- (c_3) -- (b_2) -- (c_4) -- (a_1) -- (c_2) -- (b_2) -- (c_1) -- (b_1) -- (c_4) -- (a_2) -- (c_3) -- (a_1) -- cycle;
  \draw (e_1) -- (d_1) -- (e_2) -- cycle;
\end{tikzpicture}
\]

In fact, \emph{every} cograph can be expressed in this way.
That is, cographs form the smallest category of graphs containing the singletons $\angs{\ul{1}}$ and $\angs{\ol{1}}$ that is closed under either one of the two sums and negation.
This is the meaning of the phrase \enquote{complement reducibility}.
This fact is well-known among graph-theorists \cite{MR1686154}, but
let us reprove it, as an excuse to introduce an interesting filtration by \enquote{depth}.

\subsection{Depth filtration}%
\label{sub:depthfiltration}
In this section, we will filter the category $\GG$ in two dual, ways:
first by \emph{depth}:
\[
  \GG_{\leq 1} \subset \GG_{\leq 2} \subset \cdots \GG \comma
\]
and then by \emph{co-depth}:
\[
  \GG^{\leq 1} \subset \GG^{\leq 2} \subset \cdots \GG \period
\]
These filtrations are dual, but they are also related by a shift.
Restricting these filtrations to $\DD$ and $\EE$, we obtain filtrations $\DD_{\leq \bullet}$ and $\EE^{\leq \bullet}$ of these categories.
(There are also filtrations $\DD^{\leq \bullet}$ and $\EE_{\leq \bullet}$.
These turn out to have less relevance for us.)

For every graph $\Gamma = (V,E)$,
let $\ul{R}$ be the equivalence relation generated by $E$, and
let $\ol{R}$ be the equivalence relation generated by the negation $\neg E$.
We may say that vertices $x$ and $y$ are \emph{connected} iff $x \ul{R} y$, and
that $x$ and $y$ are \emph{co-connected} iff $x \ol{R} y$.
Note that $\ol{R}$ is \emph{not} the negation of $\ul{R}$:
to be connected is to be co-connected in the negation;
there may be pairs of vertices that are both connected and co-connected.%
\footnote{This language is not ideal.
Are there better terms?}
We thus define the sets of \emph{connected} and \emph{co-connected components}:
\[
  \ul{\sigma}\Gamma \coloneq V/\ul{R} \comma \andeq \ol{\sigma}\Gamma \coloneq V/\ol{R} \period
\]
If $\ul{\sigma}\Gamma$ is a singleton, then $\Gamma$ is \emph{connected};
if $\ol{\sigma}\Gamma$ is a singleton, then $\Gamma$ is \emph{co-connected}.

The graph $P_4$ is both connected and co-connected.

At the other extreme, if $\Gamma = (V,E)$ is a \emph{cograph}, then vertices $x$ and $y$ are connected iff
there exists $w \in V$ such that both $(x,w)$ and $(w,y)$ are edges of $\Gamma_{\textit{refl}}$.
Dually, $x$ and $y$ are co-connected iff
there exists $w \in V$ such that neither $(x,w)$ nor $(w,y)$ are edges of $\Gamma_{\textit{irr}}$.

\begin{lem}
  Every cograph is either connected or co-connected.
\end{lem}

\begin{proof}
  Assume that $\Gamma$ is a cograph that is not co-connected,
  and let $x$ and $y$ be vertices that are not co-connected;
  that is, for every vertex $w$, either $(x,w)$ is an edge or $(w,y)$ is an edge of $\Gamma_{\textit{irr}}$.
  In particular, $(x,y)$ is an edge of $\Gamma_{\textit{irr}}$.

  Now let $u$ and $v$ be vertices;
  we now use the cograph condition again to show that they are connected.
  Either $(u,x)$ is an edge or $(u,y)$ is an edge,
  and either $(v,x)$ is an edge or $(v,y)$ is an edge.
  So either $u$ and $v$ are connected through $x$, or they are connected through $y$, \emph{or}
  they are the endpoints of a copy of $P_4$ of the form $(u,x,y,v)$ or $(u,y,x,v)$.
  Hence the $P_4$-freeness ensures a connection between $u$ and $v$.
\end{proof}

In order to define our depth filtration, we introduce the \emph{paw} cographs%
\footnote{The usual \emph{paw graph} is $\paw_4$.}
\[
  \paw_k \coloneq \angs{(((\ul{1} \rightsum \ul{1}) \leftsum \ul{1}) \rightsum \ul{1}) \leftsum \cdots \ul{1}} \period
\]
In other words, $V(\paw_k)$ is the set $\{1,\dots,k\}$, and for any $1 \leq i < j \leq k$, we have an edge $(i,j)$ iff $j$ is even.
The cographs $\paw_{2n}$ are connected, and the cographs $\paw_{2n-1}$ are co-connected.

We may also have the \emph{co-paw} cographs
\[
  \paw^k \coloneq (\neg \paw_k)_{\textit{irr}} \period
\]
In $\paw^k$, if $i < j$, then there is an edge $(i,j)$ iff $j$ is odd.
The cographs $\paw^{2n+1}$ are connected, and the cographs $\paw^{2n}$ are co-connected.

\begin{table}
  \label{table:pawns}
  \centering
  \begin{tabular}{cc}
    $n$ & $\paw_n$ \\
    \hline
    $1$ &
    \begin{tikzpicture}[baseline=0]
      \tikzstyle{vertex}=[circle,fill=black,minimum size=4pt,inner sep=0pt]
      \node[vertex] (v_2) at (0,0) {};
      \node[fit=(current bounding box),inner sep=1mm]{};
    \end{tikzpicture}
    \\
    $2$ &
    \begin{tikzpicture}[baseline=0]
      \tikzstyle{vertex}=[circle,fill=black,minimum size=4pt,inner sep=0pt]
      \node[vertex] (v_1) at (-1,0) {};
      \node[vertex] (v_2) at (0,0) {};
      \draw (v_1) -- (v_2) -- cycle;
      \node[fit=(current bounding box),inner sep=1mm]{};
    \end{tikzpicture}
    \\
    $3$ &
    \begin{tikzpicture}[baseline=0]
      \tikzstyle{vertex}=[circle,fill=black,minimum size=4pt,inner sep=0pt]
      \node[vertex] (v_1) at (-1,0) {};
      \node[vertex] (v_2) at (0,0) {};
      \node[vertex] (v_3) at (1,0) {};
      \draw (v_1) -- (v_2) -- cycle;
      \node[fit=(current bounding box),inner sep=1mm]{};
    \end{tikzpicture}
    \\
    $4$ &
    \begin{tikzpicture}[baseline=0]
      \tikzstyle{vertex}=[circle,fill=black,minimum size=4pt,inner sep=0pt]
      \node[vertex] (v_1) at (-1,-0.5) {};
      \node[vertex] (v_2) at (-1,0.5) {};
      \node[vertex] (v_3) at (0,0)  {};
      \node[vertex] (v_4) at (1,0)  {};
      \draw (v_1) -- (v_2) -- (v_3) -- (v_1) -- cycle;
      \draw (v_3) -- (v_4) -- cycle;
      \node[fit=(current bounding box),inner sep=1mm]{};
    \end{tikzpicture}
    \\
    $5$ &
    \begin{tikzpicture}[baseline=0]
      \tikzstyle{vertex}=[circle,fill=black,minimum size=4pt,inner sep=0pt]
      \node[vertex] (v_1) at (-1,-0.5) {};
      \node[vertex] (v_2) at (-1,0.5) {};
      \node[vertex] (v_3) at (0,0) {};
      \node[vertex] (v_4) at (1,0) {};
      \node[vertex] (v_5) at (2,0) {};
      \draw (v_1) -- (v_2) -- (v_3) -- (v_1) -- cycle;
      \draw (v_3) -- (v_4) -- cycle;
      \node[fit=(current bounding box),inner sep=1mm]{};
    \end{tikzpicture}
    \\
    $6$ &
    \begin{tikzpicture}[baseline=0]
      \tikzstyle{vertex}=[circle,fill=black,minimum size=4pt,inner sep=0pt]
      \node[vertex] (v_1) at (-1.155,0) {};
      \node[vertex] (v_2) at (-0.577,1) {};
      \node[vertex] (v_3) at (-0.577,0.333) {};
      \node[vertex] (v_4) at (-0.577,-1) {};
      \node[vertex] (v_5) at (0,0) {};
      \node[vertex] (v_6) at (1,0) {};
      \draw (v_5) -- (v_2) -- (v_1) -- (v_3) -- (v_5) -- (v_1) -- (v_4) -- (v_5) -- (v_6) -- cycle;
      \draw (v_2) -- (v_3) -- cycle;
      \node[fit=(current bounding box),inner sep=1mm]{};
    \end{tikzpicture}
    \\
  \end{tabular}
  \caption{The first few cographs $\paw_n$}
\end{table}

A \emph{paw of depth $k$} in a graph $\Gamma$ is an embedding $x \colon \paw_k \inclusion \Gamma_{\textit{irr}}$ as an induced subgraph.
In other words, it's a sequence $x_1, \dots, x_k$ of vertices such that $x_1 \neq x_2$ and if $i<j$, then there is an edge $(x_i,x_j)$ iff $j$ is even.
(Note that these vertices are automatically pairwise distinct.)
The \emph{depth} of a vertex $w$ is the supremum (possibly $\infty$) of those $k$ for which there is a paw $x$ of depth $k$ with $w = x_1$.
The \emph{depth} of $\Gamma$ itself is the supremum (possibly $\infty$ or $0$) of the depths of its vertices.
In other words, $\Gamma$ is of depth $\leq k$ iff $\Gamma_{\textit{irr}}$ is \emph{$\paw_{k+1}$-free}
(\ie, it does not contain a paw of depth $k+1$).

Dually, a \emph{co-paw of depth $k$} in $\Gamma$ is an embedding $y \colon \paw^k \inclusion \Gamma_{\textit{irr}}$ as an induced subgraph.
In other words, it's a sequence $y_1, \dots, y_k$ of vertices such that $y_1 \neq y_2$, and if $i<j$, then there is an edge $(y_i,y_j)$ iff $j$ is odd.
The \emph{co-depth} of a vertex $w$ is the largest $k$ for which there is a paw $y$ of depth $k$ with $w = y_1$.
The \emph{co-depth} of $\Gamma$ is the supremum of the co-depths of its vertices.
In other words, $\Gamma$ is of co-depth $\leq k$ iff it does not contain a co-paw of depth $k+1$.
Thus the co-depth of $\Gamma$ is the depth of $\neg \Gamma$.

The empty cograph is of depth and co-depth $0$.
A singleton (\ie, either $\angs{\ul{1}}$ or $\angs{\ol{1}}$) is of depth and co-depth $1$.

The cographs of depth $\leq 1$ are those for which there is no edge between distinct vertices.
Dually, the finite cographs of co-depth $\leq 1$ are the cographs for which there is an edge between every pair of distinct vertices.

The co-paw $\paw^k$ has depth $k-1$, which by symmetry ensures that
depth and co-depth never differ by more than $1$.

\begin{lem}
  Let $\Gamma$ be a cograph containing at least two vertices.
  It is connected iff it is of even depth, iff it is of odd co-depth.
  It is co-connected iff it is of odd depth, iff it of even co-depth.
  (We take $\infty$ be both even and odd.)
\end{lem}

\begin{proof}
  Let $\Gamma$ be connected, and
  let $x \colon \paw_{2n+1} \inclusion \Gamma$ be a paw of odd depth.
  Connectedness supplies us with a vertex $x_{2n+2}$ such that
  both $(x_{2n},x_{2n+2})$ and $(x_{2n+1},x_{2n+2})$ are edges of $\Gamma_{\textit{refl}}$.
  For $i \leq 2n-1$, we now encounter a $P_4$ as $(x_{2n+1},x_{2n+2},x_{2n},x_i)$, which ensures that each $(x_i,x_{2n})$ is an edge.
  This extends $x$ to a paw of depth $2n+2$.

  Let $\Gamma$ be co-connected, and
  let $x \colon \paw_{2n} \inclusion \Gamma$ be a paw of even depth.
  Co-connectedness supplies us with a vertex $x_{2n+1}$ such that
  neither $(x_{2n-1},x_{2n+1})$ nor $(x_{2n},x_{2n+1})$ is an edge of $\Gamma_{\textit{irr}}$.
  There are also no edges $(x_i,x_{2n+1})$ for $i \leq 2n-2$:
  if there were, we'd encounter a $P_4$ as $(x_{2n-1},x_{2n},x_i,x_{2n+1})$.
  this extends $x$ to a paw of depth $2n+1$.

  Duality supplies the remaining statements.
\end{proof}

Let $\angs{\lambda}$ and $\angs{\lambda'}$ be nonempty cographs of depths $d$ and $d'$ and co-depths $c$ and $c'$, respectively.
If $d>d'$, then the depth of $\angs{\lambda \leftsum \lambda'}$ is $d+1$ if $\angs{\lambda}$ is connected and $d$ otherwise.
If $c>c'$, then the co-depth of $\angs{\lambda \leftsum \lambda'}$ is $c+1$ if $\angs{\lambda}$ is co-connected and $c$ otherwise.

\begin{lem}
  Any finite cograph that is both connected and co-connected is a singleton.
\end{lem}

\begin{proof}
  Assume that $\Gamma$ is a cograph that is both connected and co-connected,
  and assume that $\Gamma$ has at least two distinct vertices.
  Then the depth of $\Gamma$ is both even and odd.
\end{proof}

Let $\GG_{\leq k} \subset \GG$ be the full subcategory of cographs of depth $\leq k$, and
let $\GG^{\leq k} \subset \GG$ be the full subcategory of cographs of co-depth $\leq k$.
We thus obtain two filtrations $\GG_{\leq \bullet}$ and $\GG^{\leq \bullet}$, which 
are interleaved:%
\footnote{This interleaving gives the filtration a \enquote{striped} quality that seems to be characteristic of twofold algebraic structures.}
\[
  \GG_{\leq (k-1)} \subset \GG^{\leq k} \subset \GG_{\leq (k+1)} \period
\]

\begin{prp}
  The filtrations $\GG_{\leq \bullet}$ and $\GG^{\leq \bullet}$ are exhaustive.
  In other words, $\GG$ is generated by singletons under $\rightsum$ and $\leftsum$ --
  or by $\angs{\ul{1}}$ under $\neg$ and either one of these symmetric monoidal structures.
\end{prp}

\begin{proof}
  We show that every finite cograph $\angs{\lambda}$ can be written in \emph{canonical form}.
  We regard each connected component $\angs{\mu} \in \ul{\sigma}\angs{\lambda}$ or co-connected component $\angs{\mu} \in \ol{\sigma}\angs{\lambda}$ as an induced subgraph $\angs{\mu} \subseteq \angs{\lambda}$.
  So for every $k$ we obtain a formula
  \[
    \angs{\lambda} =
      \bigrightsum_{\angs{\lambda_1} \in \ol{\sigma}\angs{\lambda}}
      \bigleftsum_{\angs{\lambda_2} \in \ul{\sigma}\angs{\lambda_1}}
      \bigrightsum_{\angs{\lambda_3} \in \ol{\sigma}\angs{\lambda_2}} \cdots
      \bigoplus_{\angs{\lambda_k} \in \ul{\sigma}\angs{\lambda_{k-1}}} \angs{\lambda_k}
  \]
  for $\angs{\lambda}$ connected, and
  \[
    \angs{\lambda} =
    \bigleftsum_{\angs{\lambda_1} \in \ul{\sigma}\angs{\lambda}}
    \bigrightsum_{\angs{\lambda_2} \in \ol{\sigma}\angs{\lambda_1}} 
    \bigleftsum_{\angs{\lambda_3} \in \ul{\sigma}\angs{\lambda_2}} \cdots
    \bigoplus_{\angs{\lambda_k} \in \ol{\sigma}\angs{\lambda_{k-1}}} \angs{\lambda_k}
  \]
  for $\angs{\lambda}$ co-connected.
  If $\angs{\lambda}$ is connected, then unless it is a singleton, the depth of any co-connected compenent is strictly smaller;
  dually, if $\angs{\lambda}$ is co-connected, then unless it is a singleton, the co-depth of any connected component is strictly smaller.
  Hence we may let $k \geq 1$ be the smallest integer such that
  the subgraphs $\angs{\lambda_k}$ are all connected and co-connected,
  hence singletons.
  (If $\angs{\lambda_{k-1}}$ is connected, then $\angs{\lambda}$ is of co-depth $k$ and depth $k+1$;
  if $\angs{\lambda_{k-1}}$ is co-connected, then $\angs{\lambda}$ is of depth $k$ and co-depth $k+1$.)
  This provides the desired formula.
\end{proof}

We call the resulting representation the \emph{canonical sum representation of $\angs{\lambda}$}.

We may pull $\GG_{\leq \bullet}$ and $\GG^{\leq \bullet}$ back to $\DD$ and $\EE$
to obtain filtrations $\DD_{\leq \bullet}$ and $\EE^{\leq \bullet}$.
Consequently:
\begin{itemize}
  \item $\DD_{\leq 1}$ is the copy of $\FF$ in $\DD$ that consists of trivial cographs $\angs{\ul{n}}$;
  \item $\EE^{\leq 1}$ is the copy of $\FF$ in $\EE$ that consists of complete cographs $\angs{\ol{n}}$;
  \item $\DD_{\leq 2}$ is the category of finite apartness relations from \cref{sub:categories}; and
  \item $\EE^{\leq 2}$ is the category of finite equivalence relations from \cref{sub:categories}.
\end{itemize}
If $k$ is even (respectively, odd), then $\DD_{\leq k+1}$ is the closure of $\DD_{\leq k}$ under $\leftsum$ (resp., under $\rightsum$).
If $k$ is even (respectively, odd), then $\EE^{\leq k+1}$ is the closure of $\EE^{\leq k}$ under $\rightsum$ (resp., under $\leftsum$).

\subsection{Dispersive \& accretive}%
\label{sub:dispersiveaccretive}
Now that we have classified the objects of the category $\GG$ of cographs,
it makes sense to study the morphisms of the category.
To begin, let's identify a factorization system.

Our construction is general.
Let $u \colon \AA \to \FF$ be a fibered category over $\FF$.
(Recall that this means that $u$ is a cartesian fibration.)
Call a map of $\AA$ \emph{dispersive} iff it is $u$-inverted -- \emph{i.e.}, is carried to an equivalence under $u$.
Call a map of $\AA$ \emph{accretive} iff it is $u$-cartesian.
Dispersive and accretive maps form wide subcategories $\AA^{\disp}, \AA_{\accr} \subseteq \AA$.
They comprise an orthogonal factorization system on $\AA$, in which
every map $\phi$ can be written uniquely as $\phi = \alpha \delta$, where $\alpha$ is accretive and $\delta$ is dispersive.

We apply this to our fibered category $\GG$ of cographs.
If $\angs{\lambda} = (V,E)$ and $\angs{\lambda'} = (V,E')$ are two cographs with the same set $V$ of vertices, and if $E \subseteq E'$, then the identity is dispersive as a map $\angs{\lambda} \to \angs{\lambda'}$.
Up to isomorphism, these are the only dispersive maps.
The trivial cographs are initial among dispersive maps;
the complete cographs are terminal among dispersive maps.

If $\angs{\lambda} = (V,E)$ and $\angs{\mu} = (W,F)$ are cographs, then a map $f \colon \angs{\lambda} \to \angs{\mu}$ is accretive iff $E = f^{-1}(F)$.
In other words, $f$ is accretive when $(x,y) \in E$ iff $(f(x),f(y)) \in F$.

An accretive injection in $\GG$ is the inclusion of an \emph{induced subgraph}.
The dispersive/accretive factorization system tilts to a factorization system in which every map $\phi$ can be written as $\phi = \iota \sigma$, where $\iota$ is an accretive injection and $\sigma$ is a surjection.

A \emph{dispersion/accretion square} in a fibered category $\AA \to \FF$ is a commutative square
\[
  \begin{tikzcd}[sep=1.5em, ampersand replacement=\&]
    I' \arrow[r, "f'"] \arrow[d,"\psi"'] \& J' \arrow[d, "\phi"] \\
    I \arrow[r, "f"']  \& J 
  \end{tikzcd}
  \comma
\]
in which $\phi$ and $\psi$ are dispersive, and $f$ and $f'$ are accretive.
Any such square is automatically a pullback in $\AA$.

The simplest nontrivial example in $\GG$ is the dispersion/accretion square
\[
  \begin{tikzcd}[sep=1.5em, ampersand replacement=\&]
    \angs{\ul{3}} \arrow[r, "f'"] \arrow[d,"\psi"'] \& \angs{\ul{2}} \arrow[d, "\phi"] \\
    \angs{\ul{2} \rightsum \ul{1}}  \arrow[r, "f"']  \& \angs{\ul{1} \rightsum \ul{1}} 
  \end{tikzcd}
  \period
\]
Dispersion/accretion squares in $\GG$ in which the accretive maps are each surjections are automatically pushouts.

If $f \colon \angs{\lambda} \to \angs{\mu}$ is an accretive map of cographs, then the same map on vertices defines a map $\neg \angs{\lambda} \to \neg \angs{\mu}$ on the negations.
On the other hand, if $f \colon \angs{\lambda} \to \angs{\lambda'}$ is a dispersive map, then the inverse of the set map defines a map $\neg \angs{\lambda'} \to \neg \angs{\lambda}$.
This suggests that the functoriality of $\neg$ mixes a covariant and contravariant functoriality.
To make this precise, we need to recall some basic facts about span categories.

\subsection{Spans}%
\label{sub:spans}
Assume that $\CC$ is a category equipped with two wide subcategories $\CC_{\dagger}, \CC^{\dagger} \subset \CC$, whose morphisms we call \emph{ingressive} and \emph{egressive}, respectively.
Let us assume that pullbacks of ingressive maps along egressive maps exist and are ingressive, and dually that pullbacks of egressive maps along ingressive maps exist and are egressive.
The resulting pullback squares we call \emph{ambigressive}.

Then we may form the \emph{span} category $\Span(\CC;\CC_{\dagger},\CC^{\dagger})$ or simply $\Span(\CC)$ for brevity, in which the objects are the objects of $\CC$, and a morphism from $X \in \CC$ to $Y \in \CC$ is a span
\[
  X \ot U \to Y
\]
in which the backward map $U \to X$ is egressive and the forward map $U \to Y$ is ingressive.
Composition is done by forming ambigressive pullback squares.
Elsewhere \cite{MR3558219}, we have called this the \emph{effective Burnside category} $\AA^{\eff}(\CC,\CC_{\dagger},\CC^{\dagger})$, and in other sources, it is called the category of \emph{correspondences}.

The formation of the span category is right adjoint to the formation of the twisted arrow category.
More precisely, for any category $A$, 
we consider the right fibration
\[
  (s,t) \colon \TwArr(A) \to A \times A^{\op}
\]
that corresponds to the functor $\Map \colon A^{\op} \times A \to \An$.
Declare a map of $\TwArr(A)$ ingressive if it is $t$-cartesian or, equivalently, $s$-inverted.
Dually, declare a map $\TwArr(A)$ egressive if it is $s$-cartesian or, equivalently, $t$-inverted.
With these definitions, ambigressive pullbacks exist in the triple
\[
  \TwArr(A) = (\TwArr(A), \TwArr(A)_{\dagger}, \TwArr(A)^{\dagger}) \comma
\]
and the category $\Span(\TwArr(A))$ is the untwisted arrow category $\Arr(A)$.
If $\CC$ is a triple as above, then let
\[
  \Fun^{\trip} (\TwArr(A), \CC)
\]
denote the category of functors $\TwArr(A) \to \CC$ that preserve ambigressive squares.
The basic result now is that
\[
  \Fun^{\trip} (\TwArr(A), \CC) = \Fun(A, \Span(\CC)) \period
\]
The unit $A \to \Span(\TwArr(A)) = \Arr(A)$ is the assignment $a \mapsto \id_a$.

Here's a particular example of the span construction in action.
Let $\AA \to \FF$ be a fibered category.
The \emph{vertical opposite} of $\AA$ is the category
\[
  \AA^{\vop} = \Span(\AA;\AA_{\accr},\AA^{\disp}) \comma
\]
which is a category fibered over $\FF$,
whose fiber over $\angs{n}$ is the opposite of the fiber $\AA_n$.
(This is a general trick \cite{MR3746613}.)
With this notation in hand, we may now return to the duality provided by negation.

\begin{prp}
  \label{sub:duality}
  The operation $\angs{\lambda} \mapsto \neg \angs{\lambda}$ is an equivalence $\GG^{\vop} \equivalence \GG$ over $\FF$.
  This is an equivalence of twofold symmetric monoidal categories $(\GG^{\vop},\rightsum,\leftsum) \to (\GG,\leftsum,\rightsum)$.
  For every $k \geq 0$ (or $k = \infty$), this equivalence restricts to equivalences $\GG_{\leq k}^{\vop} \equivalence \GG^{\leq k}$ and $\DD_{\leq k}^{\vop} \equivalence \EE^{\leq k}$.
\end{prp}

\subsection{Comments \& questions}%
\label{sub:comments-combinatorics}
Cographs form an exceptional class of graphs.
They can be characterized in many different ways, some of which we have witnessed.
Here's another, which we found striking:
an irreflexive graph is a cograph iff, within every induced subgraph, every maximal clique intersects every maximal coclique (necessarily in a single vertex).

The category $\DD_{\leq 2}$ appears in Richard Borcherds' work on quantum vertex algebras \cite{MR1865087} in a capacity very similar to its role here.
Borcherds uses it to give a very clean (he even says \enquote{trivial}) categorical definition of vertex algebras using the two induced symmetric monoidal structures on $\Fun(\DD_{\leq 2},\VV)$.
In retrospect, our framework for defining factorization algebras seems to be a sort of categorification of Borcherds' theory.
Emily Cliff \cite{MR4426317} used Borcherds' categorical tools to turn vertex algebra information into factorization algebra information and analyze the lossiness of this process.
It would be interesting to understand how her work and the present work relate.

We have briefly alluded to the theory of \emph{twofold symmetric monoidal categories} $C$, in which one has a pair of tensor products, $\lefttensor$ and $\righttensor$, a common unit, and an intertwiner
\[
  (U \righttensor V) \lefttensor (X \righttensor Y) \to (U \lefttensor X) \righttensor (V \lefttensor Y) \period
\]
These structures play a significant role in this paper.
In effect, a twofold symmetric monoidal category is a commutative monoid in the $2$-category $\categ{Mon}^{(\infty),\textit{oplax}}\categ{Cat}$ of symmetric monoidal categories and normal oplax symmetric monoidal functors.%
\footnote{The word \emph{normal} in this context means that the image of the unit is a unit.}
Or equivalently it is a commutative monoid in the $2$-category $\categ{Mon}^{(\infty),\textit{lax}}\categ{Cat}$ of symmetric monoidal categories and normal lax symmetric monoidal functors.

In \cref{sec:twofold} we develop the elements of such a theory for using the category $\EE$, but
our framework is not completely satisfactory for reasons we discuss in \cref{sub:comments-twofold}.

In any case, the combinatorics of cographs plays a central role in our theory of isolability objects,
to which we now turn.


\section{Isolability structures}%
\label{sec:isolability}
In this section, we identify a structure that permits one to make sense of factorization algebras over a given geometric object.
The goal is to do so in terms that are independent of any particular geometric context (such as smooth manifolds, complex manifolds, supermanifolds, varieties, \etc.)
The structure we define -- what we call an \emph{isolability structure}%
\footnote{Author's note:
  At times I have also called these structures \emph{world structures}, but
  this terminology overplays a particular physical interpretation.
  In many cases, the isolability structure one wants lies not on the \enquote{worldvolume object} itself,
  but rather on an auxiliary \emph{observer stack} (\cref{sub:observer}) of objects over the worldvolume.}
-- allows us to make sense of a definition of factorization algebra on a space with coefficients in any reasonable category of sheaves.

\subsection{A basic example}%
\label{sub:basicexample}
Before the general definition, an example.
Let $X$ be a topological space.
(One should have in mind a topological manifold here;
in any case we take all our topological spaces to be compactly generated, k-ifying them when necessary.)
For each cograph $\angs{\lambda}$, let's say that a map $x \colon V\angs{\lambda} \to X$ is \emph{separating} if for every edge $(a,b) \in E\angs{\lambda}$, one has $x(a) \neq x(b)$.
In other words, a separating map is one that carries isolated vertices of $\angs{\lambda}$ to distinct points of $X$.
We write $X^{\lambda} \subset \Map(V \angs{\lambda}, X)$ for the subspace consisting of the separating maps.

For example, the subset $X^{\ul{m} \rightsum \ul{n}} \subset X^{m+n}$ consists of those $(x_1, \dots, x_{m+n})$ in which the sets $\{x_1, \dots, x_m\}$ and $\{x_{m+1}, \dots, x_{m+n}\}$ are disjoint.

If $\angs{\lambda}$ is not irreflexive, this definition is empty, but
we do obtain an interesting functor $X^{\bullet} \colon \DD^{\op} \to \Top$:
\begin{itemize}
  \item An accretive surjection $i \colon \angs{\mu} \to \angs{\lambda}$ is carried to the inclusion
    \[ \tensor*[_X]{i}{} \colon X^\lambda \inclusion X^\mu \]
    given by the equations $x(a) = x(b)$ if $i(a) = i(b)$.
    This is closed when $X$ is Hausdorff.
    Example: the accretive map $\angs{\ul{2}} \to \angs{\ul{1}}$ is carried to the diagonal
    \[
      X^{\ul{1}} = \Delta_X \inclusion X \times X = X^{\ul{2}} \period
    \]
  \item A dispersive map $j \colon \angs{\lambda'} \to \angs{\lambda}$ 
    is carried to the inclusion
    \[ \tensor*[_X]{j}{} \colon X^{\lambda} \inclusion X^{\lambda'} \]
    cut out by the inequalities $x(a) \neq x(b)$ for $(a, b) \in E\angs{\lambda}$.
    This is open when $X$ is Hausdorff.
    Example: the dispersive map $\angs{\ul{2}} \to \angs{\ul{1} \rightsum \ul{1}}$ is carried to the inclusion
    \[
      X^{\ul{1} \rightsum \ul{1}} = (X \times X) \smallsetminus \Delta_X \inclusion X \times X = X^{\ul{2}} \period
    \]
  \item An accretive injection $g \colon \angs{\mu} \to \angs{\lambda}$ is carried to the restriction 
    \[ \tensor*[_X]{g}{} \colon X^\lambda \fibration X^\mu \]
    of the projection $X^{V\angs{\lambda}} \fibration X^{V\angs{\mu}}$.
    When $X$ is a manifold, this is a fibration, sometimes called the \emph{Faddell--Neuwirth fibration}.
    Example: the two inclusions $\angs{\ul{1}} \inclusion \angs{\ul{1} \rightsum \ul{1}}$ are carried to the two projections
    \[
      X^{\ul{1} \rightsum \ul{1}} = (X \times X) \smallsetminus \Delta_X \to X = X^{\ul{1}} \period
    \]
\end{itemize}

The resulting diagram $X^{\bullet} \colon \DD^{\op} \to \Top$ is what we call an \emph{isolability topological space}.
It has a number of interesting features:
\begin{itemize}
  \item For every pushout square
    \[
      \begin{tikzcd}[sep=1.5em, ampersand replacement=\&]
        \angs{\lambda} \arrow[r, "i"] \arrow[d,"g"'] \& \angs{\mu} \arrow[d, "h"] \\
        \angs{\lambda'} \arrow[r, "i'"']  \& \angs{\mu'} 
      \end{tikzcd}
    \]
    in $\DD$ in which $i$ is surjective and $g$ is an accretive injection,
    the corresponding square
    \[
      \begin{tikzcd}[sep=1.5em, ampersand replacement=\&]
        X^{\mu'} \arrow[r] \arrow[d] \& X^{\lambda'} \arrow[d] \\
        X^{\mu} \arrow[r]  \& X^{\lambda} 
      \end{tikzcd}
    \]
    is a pullback.
    For example, $X^{\ul{1} \rightsum \ul{1}} = X^1 \times_{X^2} X^{\ul{1} \rightsum \ul{2}}$.
    This is \emph{regularity} (\cref{sub:regularity}).
  \item In the Hausdorff case, accretive surjections $i$ are carried to closed inclusions $\tensor*[_X]{i}{}$, and dispersive maps $j$ are carried to open immersions $\tensor*[_X]{j}{}$.
  \item Furthermore, the open immersion $\tensor*[_X]{j}{}$ induced by a dispersive map $j$ is the complement of the union of the closed embeddings $\tensor*[_X]{i}{}$ induced by the accretive surjections $i$ that are incompatible (in the sense of \cref{sub:categories}) with $j$.
  \item The diagram $X^{\bullet}$ carries the disconnected sum $\leftsum$ to products.
    This is \emph{additivity} (\cref{sub:additivity}).
    This condition is reasonable in a category like $\Top$,
    but for other categories of interest such as stratified spaces (\cref{sec:homotopical}), additivity is too restrictive.
  \item Our functor $X^{\bullet}$ is left Kan extended from the subcategory $\DD_{\leq 2)}$: 
    if $\angs{\lambda}$ and $\angs{\mu}$ are two irreflexive cographs, then
    the space $X^{\lambda \leftsum \mu} = X^\lambda \times X^\mu$ is covered by the spaces $X^{\nu}$ in which $\angs{\nu}$ is any finite isolation structure equipped with maps $\angs{\lambda} \to \angs{\nu}$ and $\angs{\mu} \to \angs{\nu}$.
    For example, we find that 
    \[
      X^{\paw_3} = X^{\ul{1}} \times X^{\ul{1} \rightsum \ul{1}} = X^{\ul{2} \rightsum \ul{1}} \cup^{X^{\ul{1} \rightsum \ul{1} \rightsum \ul{1}}} X^{\ul{1} \rightsum \ul{2}} \comma
    \]
    which expresses the co-transitivity of the relation $\neq$ on $X$.
    This is \emph{$2$-skeletality} (\cref{sub:skeletality}).
\end{itemize}

\subsection{Isolability objects}%
\label{sub:isolabilityobjects}
An \emph{isolability object} of a category $\XX$ is a functor $W^{\bullet} \colon \DD^{\op} \to \XX$.
The object $W^1$ is the \emph{underlying object} of $W^{\bullet}$, and $W^{\bullet}$ is an \emph{isolability structure} on $W^1$.
Isolability objects form the category $\Isol(\XX) \coloneq \Fun(\DD^{\op},\XX)$.

As we have seen (\cref{sub:basicexample}), if $X$ is a topological space, then $\angs{\lambda} \mapsto X^{\lambda}$ is an isolability topological space.
More generally, if $f \colon X \to Y$ is a continuous map of topological spaces, then we may define $X_f^{\lambda}$ as the set of maps $x \colon V\angs{\lambda} \to X$ such that $f(x)$ is separating.
This defines a functor $\TwArr(\Top) \to \Isol(\Top)$.
Equivalently (\cref{sub:spans}), this operation can be regarded as a functor $\Top \to \Span(\Isol(\Top))$ that carries a map $f \colon X \to Y$ to the span $X^{\bullet} \ot X_f^{\bullet} \to Y^{\bullet}$.

Similarly, the assignment $\angs{\lambda} \mapsto \Pi_{\infty}(X^{\lambda})$ is an additive isolability object in spaces.
As it turns out, however, this structure is too coarse for our purposes, and
we will need to incorporate stratifications for a theory of \emph{isolability spaces} that is suitable for the \emph{homotopical context} (\cref{sec:homotopical}).

\subsection{Regularity}%
\label{sub:regularity}

We say that an isolability object $X^{\bullet} \colon \DD^{\op} \to \XX$ is \emph{regular} iff
for every pushout square
\[
  \begin{tikzcd}[sep=1.5em, ampersand replacement=\&]
    \angs{\lambda} \arrow[r, "i"] \arrow[d,"g"'] \& \angs{\mu} \arrow[d, "h"] \\
    \angs{\lambda'} \arrow[r, "i'"']  \& \angs{\mu'} 
  \end{tikzcd}
\]
in $\DD$ in which $i$ is surjective and $g$ is an accretive injection,
the corresponding square
\[
  \begin{tikzcd}[sep=1.5em, ampersand replacement=\&]
    X^{\mu'} \arrow[r] \arrow[d] \& X^{\lambda'} \arrow[d] \\
    X^{\mu} \arrow[r]  \& X^{\lambda} 
  \end{tikzcd}
\]
is a pullback.

In a regular isolability object $X^{\bullet}$, the map $X^{\mu'} \to X^{\mu}$ induced by an accretive injection is pulled back from the map $X^{V\angs{\mu'}} \to X^{V\angs{\mu}}$.
If $f \colon X \to Y$ is a continuous map of topological spaces, then $X_f^{\bullet}$ is a regular isolability object. 

The majority of the examples of isolability objects we've encountered are regular,
so it appears to be quite a natural condition.
On the other hand, it doesn't appear to be necessary for a good theory of factorization algebras.

\subsection{Additivity}%
\label{sub:additivity}
An isolability object $W^{\bullet} \colon \DD^{\op} \to \XX$ is an \emph{additive} iff it carries disconnected sums (\ie, finite coproducts) to products.
Additive isolability objects can also be described simply as functors (with no conditions) on the category $\DD^{\op}_{\textit{conn}}$ of connected irreflexive cographs.

This condition is reasonable for isolability objects in topological spaces and varieties, but
it doesn't mix well with the presence of stratifications.
For example, in the homotopical context (\cref{sec:homotopical}), we will want to regard $W^2$ as stratified by the diagonal and its complement.

\subsection{Skeletality}%
\label{sub:skeletality}
Assume that the category $\XX$ has finite colimits, and let $k \geq 1$.
We may left Kan extend a functor $W^{\bullet} \colon \DD_{\leq k}^{\op} \to \XX$ to an isolability object. 
For an object $\angs{\lambda} \in \DD$, this extension is defined by the formula:
\[
  W^{\lambda} = \colim_{\angs{\mu} \in (\DD_{\leq k, \angs{\lambda}/})^{\op}} W^{\mu} \period
\]
Limit-cofinal in $\DD_{\leq k,\angs{\lambda}/}$ is the finite full subcategory consisting of dispersive maps $\angs{\lambda} \to \angs{\mu}$. 
This is why only finite colimits are needed for this extension to exist.
We call isolability objects constructed in this way \emph{$k$-skeletal}.

Thus if we define the category $\Isol_{\leq k}(\XX) \coloneq \Fun(\DD_{\leq k}^{\op}, \XX)$, then
left Kan extension identifies $\Isol_{\leq k}(\XX)$ with the full subcategory of $\Isol(\XX)$ consisting of $k$-skeletal isolability objects.
In effect, if $W^{\bullet}$ is $k$-skeletal, then $W^{\lambda}$ is the union of $W^{\mu}$ over those cographs $\angs{\mu} \supseteq \angs{\lambda}$ with the same vertices as $\angs{\lambda}$ such that $\angs{\mu}$ is $\paw_{k+1}$-free.

Every isolability object $X^{\bullet}$ thus has a \emph{skeletal filtration}
\[
  \sk_1 X^\bullet \to \sk_2 X^\bullet \to \cdots \to X^\bullet \period
\]
The $1$-skeleton is the isolability object
\[
  (\sk_1 X)^{\lambda} =
  \begin{cases}
    X^n         & \text{if } \angs{\lambda} = \angs{n} \comma \\
    \varnothing & \text{otherwise.}
  \end{cases}
\]
The $2$-skeleton is more complicated;
for example, we have the following cotransitivity formula:
\[
  (\sk_2 X)^{\paw_3} =
  X^{\ul{1} \rightsum \ul{2}} \cup^{X^{\ul{1} \rightsum \ul{1} \rightsum \ul{1}}} X^{\ul{2} \rightsum \ul{1}}
  \period
\]
The construction of \cref{sub:basicexample} carries a Hausdorff topological space $X$ to a $2$-skeletal isolability topological space $\angs{\lambda} \mapsto X^{\lambda}$. 

In the other direction, let us suppose that $\XX$ has all limits.
One can then \emph{right} Kan extend a functor $\DD_{\leq n}^{\op} \to \XX$ to an isolability object $\DD^{\op} \to \XX$.
Isolability objects arising this way are called \emph{$n$-coskeletal}, and
right Kan extension identifies $\Isol_{\leq n}(\XX)$ with the full subcategory of $\Isol(\XX)$ consisting of $n$-coskeletal isolability objects.
One thus defines a \emph{coskeletal tower}
\[
  X^\bullet \to \cdots \to \cosk_2(X)^\bullet \to \cosk_1(X)^\bullet \period
\]

The $1$-coskeleton is given by the formula
\[
  \cosk_1(X)^{\lambda} = X^n \comma
\]
where $\angs{n} = V\angs{\lambda}$.
A $1$-coskeletal isolability object is one that inverts dispersive maps.
In effect, every point of a $1$-coskeletal isolability object is isolated in exactly one way from every point -- including itself.

\subsection{Products}%
\label{sub:products}
The categorical product in the category $\Isol(\XX)$ is not the right one from our perspective:
morally, the points of $X^{\lambda} \times Y^{\lambda}$ are pairs $(x,y)$ of configurations of points in which, for every $(a,b) \in E\angs{\lambda}$, \emph{both} $x_a \neq x_b$ and $y_a \neq y_b$.
What we really want are pairs $(x,y)$ such that if $(a,b) \in E\angs{\lambda}$, then \emph{either} $x_a \neq x_b$ or $y_a \neq y_b$.
That is we want to define $(X \otimes Y)^{\lambda}$ as the union of the objects $X^{\mu_1} \times X^{\mu_2}$ in which $\angs{\mu_1}$ and $\angs{\mu_2}$ are cographs with the same vertices as $\angs{\lambda}$, in which every edge of $\angs{\lambda}$ lies in either $\angs{\mu_1}$ or $\angs{\mu_2}$.
So for our purposes, the correct product is a certain convolution.

To this end, let us define an \emph{anti-operad} structure on $\DD$
(\ie, an operad structure on $\DD^{\op}$ \cite[Df. 1.2]{MR3933393}).
For cographs $\angs{\lambda}, \angs{\mu_1}, \dots \angs{\mu_n}$,
we define a map $f \colon \angs{\lambda} \to \angs{\mu_1} \otimes \cdots \otimes \angs{\mu_n}$ as a map $f \colon V\angs{\lambda} \to V\angs{\mu_1} \times \cdots \times V\angs{\mu_n}$ such that if $(a,b) \in E\angs{\lambda}$, then for \emph{some} $i$, there is an edge $(f_i(a),f_i(b)) \in E\angs{\mu_i}$.
This is misleading notation, because there is no operation $\otimes$ on cographs that gives this.
The anti-operad $\DD_{\otimes}$ is however a symmetric \emph{promonoidal} structure \cite[Df. 1.4]{MR3933393} on $\DD$, which suffices for our purposes.

Now assume that the category $\XX$ has products and finite colimits.
Then the category $\Isol(\XX)$ has a Day convolution symmetric monoidal structure $\otimes$ corresponding to the symmetric promonoidal structure $\DD_{\otimes}$ \cite[1.6]{MR3933393}.
We call this symmetric monoidal structure the \emph{tensor product} of isolability objects.

To unpack this, let $X^\bullet$ and $Y^\bullet$ be isolability objects of $\XX$.
Then the tensor product is given by 
\[
  (X \otimes Y)^{\lambda} = \colim_{\angs{\lambda} \to \angs{\mu_1} \otimes \angs{\mu_2}} X^{\mu_1} \times X^{\mu_2} \comma
\]
where the colimit is formed over the category of triples $(\angs{\mu_1}, \angs{\mu_2}, \phi)$, where $\phi \colon \angs{\lambda} \to \angs{\mu_1} \otimes \angs{\mu_2}$.
Cofinal in this category are those triples in which the two maps $V\angs{\lambda} \to V\angs{\mu_1}$ and $V\angs{\lambda} \to V\angs{\mu_2}$ are each bijections.

For example, one has
\[
  (X \otimes Y)^{\ul{1} \rightsum \ul{1}} = (X^2 \times Y^{\ul{1} \rightsum \ul{1}}) \cup^{(X^{\ul{1} \rightsum \ul{1}} \times Y^{\ul{1} \rightsum \ul{1}})} (X^{\ul{1} \rightsum \ul{1}} \times Y^2) \period
\]

\subsection{Observer stacks}%
\label{sub:observer}


Any reasonable category of geometric objects can be embedded into a topos%
\footnote{Recall that a \enquote{topos} for us is what elsewhere is called an \enquote{$\infty$-topos} \cite{MR2522659}.}
in a way that preserves existing features of the category,
such as limits and gluing constructions.
We call an objects of a topos a \emph{stack} on that topos.
For example, if we are considering the category $\categ{Man}$ of smooth manifolds, then
we may embed it into the topos of stacks on euclidean spaces relative to the smooth topology.
Up to set-theoretical issues 
(which can all be addressed using standard tricks),
any category of schemes can be embedded into a category of stacks for one of several different topologies.
Even (compactly generated) topological spaces can be embedded into the topos of condensed (\emph{alias} pyknotic) spaces.

A big advantage of enlarging a category of geometric objects in this way is that
one can then perform a variety of constructions that might be unavailable in the smaller category.
Historically, the first example is the formation of quotients $X/G$.
More constructions become available by observing that a functor $\XX^{\op} \to \Space$ is representable by a stack iff
it \emph{satisfies descent} -- \ie, it carries colimits of $\XX$ to limits.
For example, one defines the \emph{mapping stack} $\MAP(Y,X)$ as the stack that represents the functor $T \mapsto \Map(T \times Y, X)$.
Similarly, one defines $\Obj_X$, the \emph{stack of objects over $X$}, as the stack that represents the functor that carries an object $T$ to the groupoid core $\iota(\XX_{/T \times X})$ of the overcategory;
that this is a stack is a fundamental fact about topoi, sometimes expressed with the motto \enquote{all colimits in $\XX$ are van Kampen}.

We highlight here a natural approach to defining isolability structures on certain stacks on $\XX$.

Let $x \in X(S) = \Map(S,X)$ and $y \in X(T) = \Map(T,X)$ be two points of a stack $X$.
We say that $x$ and $y$ are \emph{disjoint} iff
the path stack $S \times_X T$ is empty.

Let $X$ be a stack, and consider a stack $O_X$ over the stack $\Obj_X$.
Thus a point of $O_X$ is an object over $X$ along with some additional structure.
We think of the stack $O_X$ as telling us what sort of observables we are considering.
For example, we could let $O_X = X$ itself, in which case we are considering only point observables.
The Hilbert scheme of a projective variety $X$ (say) is a more interesting example to bear in mind;
in this case, we are considering extended observables along closed subvarieties.

There are many ways to equip $O_X$ with an isolability structure.
Perhaps the easiest is the one in which two objects over $X$ are isolated iff they do not intersect.
More precisely, for every cograph $\angs{\lambda}$ and every $T \in \XX$, define
\[
  O_X^{\lambda}(T) \coloneq \{Z \in O_X(T)^{V\angs{\lambda}} : (a,b) \in E\angs{\lambda} \implies Z_a \times_{T \times X} Z_b = \varnothing\} \period
\]
In this formula we are implicitly applying our forgetful map $O_X \to \Obj_X$ to $Z_a$ and $Z_b$ in order to form this fiber product.
The isolability stack $O_X^{\bullet}$ that results is both additive and local.

When $O_X = X$ itself, we end up with an isolability structure on $X$ in which
points $x,y \in X(T)$ are isolated in this structure iff their equalizer $\enquine{x = y}$ is empty.
The resulting isolability stack $X^\bullet$ is $2$-skeletal.
However, in general this isolability structure on $O_X$ is not skeletal at all.

\subsection{Hilbert schemes as observer stacks}%
\label{sub:Hilbertschemes}
Let $k$ be a field, and let $X$ be a $k$-variety.
Write $\Hilb_{X/k}$ for the Hilbert algebraic space,
whose $T$-points are the closed subvarieties of $T \times_k X$ that are proper and fppf over $T$.
We can now define $\Hilb^{\bullet}_{X/k}$ as an isolability stack using the recipe of \cref{sub:observer}.
Thus the $k$-points of $\Hilb^{\lambda}_{X/k}$ are tuples $(Z_a)_{a \in V\angs{\lambda}}$ of subvarieties of $X$ such that if $(a,b) \in E\angs{\lambda}$, then $Z_a$ does not intersect $Z_b$.

This isolability structure is $2$-skeletal when $X$ is a curve, but not more generally.

\subsection{$\Div^1$ as an observer stack}%
\label{sub:FFcurve}
As an illustration of the scope of the technology, here is an isolability space that arises naturally in $p$-adic geometry in the style of \cite{MR4446467}.

Write $\categ{Perfd}$ for the category of perfectoid spaces, and
write $\categ{Perf}$ for the category of perfectoid spaces of characteristic $p$.
The tilting functor $U \mapsto U^{\flat}$ is a functor $\categ{Perfd} \to \categ{Perf}$.
If $U$ is perfectoid, then tilting defines an equivalence $\categ{Perfd}_{/U} \equivalence \categ{Perf}_{/U^{\flat}}$ \cite[7.1.4]{MR4446467}.
Consequently, the tilting functor is a right fibration with discrete fibers.
The fiber over a perfectoid space $T$ of characteristic $p$ is the set of \emph{untilts} of $T$.

Let $\XX$ be the topos of proétale stacks on $\categ{Perf}$.
Given an analytic adic space $V$ over $\Spa \ZZ_p$, the right fibration $\categ{Perfd}_{/V} \to \categ{Perf}$ straightens to an object $V^{\diamond} \in \XX$.
For example, one has $\Spd \QQ_p \coloneq (\Spa \QQ_p)^{\diamond}$:
if $T$ is a perfectoid space of characteristic $p$, then the $T$-points of $\Spd \QQ_p$ are the characteristic $0$ untilts of $T$.
We obtain a relative version over $S \in \Perf$ by forming $Y_S^{\diamond} \coloneq S \times \Spd \QQ_p$, from which we form the quotient $X_S^{\diamond} \coloneq Y_S^{\diamond}/(\phi^{\ZZ} \times \id)$.
The object $X_S^{\diamond}$ is the \emph{Fargues--Fontaine curve}, regarded as a \enquote{diamond}.

In \cite{geometrization}, Fargues and Scholze also define a \enquote{mirror curve} $\Div^1 \coloneq \Spd \QQ_p/\phi^{\ZZ}$.
The notation here reflects the fact that an $S$-points of $\Div^1$ defines a degree $1$ Cartier divisor $D_S$ on $X_S$.
In their geometrization of local Langlands, Hecke operations arise via modifications of vector bundles on the Fargues--Fontaine curve which occur along such divisors \cite[Chapter VI]{geometrization}.
So the natural observer stack (\cref{sub:observer}) for the Fargues--Fontaine curve is $\Div^1$.

Using the recipe from \cref{sub:observer}, we obtain the isolability object $\Div^{\bullet}$.
Explicitly, for a cograph $\angs{\lambda}$, the $T$-points of $\Div^{\lambda}$ are collections $(T^{\sharp}_a)_{a \in V\angs{\lambda}}$ of characteristic $0$ untilts of $T$ up to Frobenius such that if $(a,b) \in E\angs{\lambda}$, then $D_{T_a} \times_{X_T} D_{T_b} = \varnothing$.

\subsection{Comments \& questions}%
\label{sub:comments-isolability}
A natural question is whether there's any significant difference between an isolability object and its Ran space.
We may think of an isolability space as a suitable \emph{input} for a Ran space construction.
Certainly in many situations, both the isolability object and the Ran space attempt to assemble configuration spaces of points.
And indeed, in these situations, the data of the isolability space is roughly equivalent to the data of its Ran space as a commutative monoid.

On the other hand, we do not understand the Ran space construction in very much generality;
in this paper, we consider the Ran space only in the \emph{homotopical context} (\cref{sub:ran}).
There, it turns out to be a pullback of a twisted form of the total space of the isolability space.
For isolability objects that are not $2$-skeletal, thus Ran space contains \emph{strictly less} information.
It may be beneficial to try to develop a notion of a Ran space for isolability structures on observer isolability stacks \cref{sub:observer}.

Another natural question is how one ought to think about the locality condition.
We do not have an explanation for it from any sort of first principles, and
we don't know how seriously to take it.
We only note that it is satisfied in the most interesting examples.

A propos of \cref{sub:products}, one wonders how (or whether?) the operad structure on $\DD^{\op}$ is an emergent structure just from the universal properties of $\DD$ as a twofold symmetric monoidal category.


\section{Isolability spaces \& the homotopical context}%
\label{sec:homotopical}
The observables of a \emph{topological} quantum field theory constitute a \emph{locally constant factorization algebra}.
This is the \emph{homotopical context} discussed in the introduction.
Locally constant factorization algebras are precisely the same as factorization algebras on what we call \emph{isolability spaces}.
These are not quite the same thing as isolability objects in spaces, because one needs to incorporate stratifications.

Using structures on cographs, we give a combinatorial model for isolability spaces attached to the real line and to euclidean spaces more generally.

\subsection{Stratifications}%
\label{sub:Stratifications}
Stratifications arise naturally in our story.
For instance, in our motivating example \cref{sub:basicexample},
the topological space $X^n$ admits a natural stratification by the various weak diagonals and their complements.
So that we can describe the interaction between stratifications and isolability structures, we recall some basic elements of the language.

Let $P$ be a poset.
We endow $P$ with the \emph{Alexandroff topology}, in which $U \subseteq P$ is open iff $(p \leq q) \wedge (p \in U) \implies q \in U$.
Now if $X$ is a topological space, then a \emph{$P$-stratification} of $X$ is a continuous map $f \colon X \to P$;
the fiber $X_p$ is the $p$-th \emph{stratum}.
A stratified map $(f/\phi) \colon (X/P) \to (Y/Q)$ now consists of a continuous map $f \colon X \to Y$ and a monotonic map $\phi \colon P \to Q$ that enjoy the expected compatibility.

For example, the interval $I = [0,1]$ is stratified over the poset $\{0<1\}$ by the ceiling function.
The strata are $\{0\}$ and the half-open interval $(0,1]$.
A stratified map $I \to (X/P)$ is called an \emph{exit path} in $(X/P)$.
More generally, the standard topological simplices $|\Delta^n| \subset \RR^{n+1}$ are stratified over the poset $[n] \coloneq \{0<\cdots<n\}$ by the map $t \mapsto \max\{i \in [n] : t_i \neq 0\}$.
Stratified maps from these are stratified (higher) homotopies.

There is a well-behaved homotopy theory of these objects.
A \emph{$P$-stratified space} is a category $X$ along with a conservative functor $\phi \colon X \to P$;
that is, for every $p \in P$, the fiber $X_p$ of $\phi$ over $p$ is a space.
This space is called the $p$-th \emph{stratum} of $X$.
Additionally, given $(p,q)$ in $P$ with $p \leq q$, we can speak of the \emph{link} of $X$ over $(p, q)$ of $P$, which is the space of sections $\{p \leq q\} \to X$;
this too is a space.

Stratified topological spaces have homotopy types, which we call \emph{stratified spaces} \cite{MR4732611}:
as long as the stratified topological space $X/P$ is sufficiently nice
(fibrant in Haine's model structure \cite{MR4732611}, conically stratified in the sense of Lurie \cite{HA}),
we can define a category $\Pi(X/P)$ in the following manner.
The objects are points of $X$, the morphisms are exit paths, and the higher morphisms are stratified homotopies.
That is, this category is defined so that a functor $[n] \mapsto \Pi(X/P)$ is precisely a stratified map $(|\Delta^n|/[n]) \to (X/P)$.
The continuous map $X \to P$ induces a conservative functor $\Pi(X/P) \to P$.

In other words, $\Pi(X/P)$ is a $P$-stratified space.
With a natural notion of \emph{weak equivalence} of stratified topological spaces,
the assignment $(X/P) \mapsto \Pi(X/P)$ is then an equivalence of homotopy theories \cite{MR4732611}.

For our purposes here, we do not want to carry around the poset over which we will be stratifying as extra structure. 
Accordingly, we simply declare that a \emph{stratification} is a category in which every endomorphism is an equivalence.
(Elsewhere, these are called \emph{layered categories} or \emph{EI categories}.)
Every such category $X$ maps conservatively to the poset $P_X$ of equivalence classes of objects of $X$ in which $x \leq y$ iff there exists a morphism $x \to y$ in $X$.
Thus a stratification $X$ in our sense is a $P_X$-stratified space whose strata and links are all connected.

We write $\Strat$ for the full subcategory of $\Cat$ consisting of stratifications.

Certain categorical properties have interpretations in the context of stratified homotopy theory.
A \emph{closed immersion} is a sieve inclusion $Z \inclusion X$ -- \ie, a full faithful inclusion in which, for every edge $x \to y$ in $X$, if $y \in Z$, then $x \in Z$.
Dually, an \emph{open immersion} in $X$ is a cosieve inclusion $U \inclusion X$.
Thus open immersions have closed complements, and \emph{vice versa}.
\emph{Locally closed immersions} are fully faithful inclusions $W \inclusion X$ such that every object through which a morphism of $W$ factors is itself an object of $W$.
All of these conditions are pulled back from the analogous ones at the level of posets $P_X$.

\subsection{Isolability spaces}%
\label{sub:isolabilityspaces}
An \emph{isolability space} is defined to be a functor $\DD^{\op} \to \Strat$.
One way to specify an isolability space is to take an \emph{isolability topological space} $(X^{\bullet}/P^{\bullet})$ over an \emph{isolability poset} $P^\bullet$,
and form the stratified homotopy type degreewise.

Perhaps the most important example for us will be the \emph{canonical isolability poset}
\[
  \angs{\lambda} \mapsto K^{\lambda} \coloneq \DD^{\disp}_{\angs{\lambda}/} \comma
\]
where the map $K^{\lambda} \to K^{\mu}$ attached to $\phi \colon \angs{\mu} \to \angs{\lambda}$ carries a dispersive map $\delta \colon \angs{\lambda} \to \angs{\lambda'}$ to the dispersive map appearing in the dispersive/accretive factorization of the composite $\delta\phi$.
The construction of \cref{sub:isolabilityobjects} gives, for every continuous map of topological spaces $f \colon X \to Y$, an isolability topological space $X_f^\bullet$.
When $Y$ is Hausdorff, we may stratify each $X_f^{\lambda}$ over $K^{\lambda}$ by sending $x \in X_f^{\lambda}$ to the dispersive map $\angs{\lambda} \to \angs{\lambda'}$ in which $(i,j) \in E\angs{\lambda'}$ iff $f(x_i) \neq f(x_j)$. 
(This map actually lands in the fragment $K^{\lambda}_{\leq 2} \coloneq \DD^{\disp}_{\angs{\lambda}/} \times_{\DD} \DD_{\leq 2}$.)

In any case, we form the stratified homotopy type degreewise, giving an isolability space
\[
  \angs{\lambda} \mapsto \Pi(X_f^{\lambda}/K^{\lambda}) \period
\]
The assignment $f \mapsto \Pi(X_f^{\bullet}/K^{\bullet})$ is a functor $\TwArr(\Haus) \to \Isol(\Strat)$ --
or, equivalently, a functor $\Haus \to \Span(\Isol(\Strat))$.

Factorization algebras on the isolability stratification $\Pi(X^{\bullet}/K^{\bullet})$ will turn out to be the same as \enquote{locally constant factorization algebras}.
Thus isolability spaces are the basic \enquote{spaces} of the \emph{homotopical context} for factorization algebras.
We regard $\Pi(X^{\bullet}/K^{\bullet})$ as an \emph{isolability homotopy type} attached to the topological space $X$.

An isolability space $X^{\bullet}$ is \emph{Hausdorff} iff:
\begin{itemize} 
  \item every accretive surjection $i$ is carried to a closed immersion $\tensor*[_X]{i}{}$ of stratified spaces, and
  \item every dispersive map $j$ is carried to an open immersion $\tensor*[_X]{j}{}$ of stratified spaces.
\end{itemize}
It is further said to be \emph{separated} iff, in addition:
\begin{itemize}
  \item the open immersion $\tensor*[_X]{j}{}$ induced by a dispersive map $j$ is complementary to the union of all the closed immersions $\tensor*[_X]{i}{}$ induced by the accretive surjections $i$ that are incompatible with $j$.
\end{itemize}

\subsection{Para-isolability \& envelopes}%
\label{sub:envelopes}

In order to specify an isolability object $X^\bullet$, one wants to specify the objects $X^{\lambda}$. 
Morally, the $T$-points are collections $x = (x_i)_{i \in V\angs{\lambda}}$ of $T$-points of $X$ such that if $(i,j)$ is an edge, then $x_i$ and $x_j$ are \enquote{distant}.
It often happens (\eg, \cref{sub:1cographs}) that it is easier to specify the collections $x$ such that $x_i$ and $x_j$ are \enquote{distant} \emph{if and only if} $(i,j)$ is an edge.
We have a recipe for turning this data into an isolability object.

So now let $p \colon \XX \to \DD$ be a functor, and
let us assume that $p$ is only a cartesian fibration over the accretives.
That is, for every accretive morphism $g \colon \angs{\lambda} \to \angs{\mu}$ of cographs and every $y$ lying over $\mu$, there exists a $p$-cartesian morphism $x \to y$ lying over $g$.
In this case, we may call $\XX$ a \emph{para-isolability category}.
This is tantamount to a functor $X$ from $\DD^{\op}$ to profunctors that are honest functors on accretives.

Morally, the \emph{envelope} of $X$ is the isolability category $\Env(X) \colon \DD^{\op} \to \Cat$ that carries $\angs{\lambda}$ to the category of pairs $(\phi \colon \angs{\lambda} \to \angs{\mu}, x)$ in which $\phi$ is a dispersive map in $\DD$, and $x \in X_{\mu}$.
In other words, $\Env(X)$ will carry $\angs{\lambda}$ to $P^{\lambda} \times_{\DD} \XX$.
The functoriality works like this: if $\eta \colon \angs{\lambda'} \to \angs{\lambda}$ is any map of $\DD$, then
the induced functor $\Env(\DD)^{\lambda} \to \Env(\DD)^{\mu}$ carries the object $(\phi, x)$ to $(\psi,y)$, where $\phi \eta = \psi \theta$ for $\psi$ dispersive, $\theta$ accretive, and $y = \tensor*[_X]{\theta}{}(x)$.

To make this more precise, a piece of notation:
if $f \colon A \to C$ and $g \colon B \to C$ are functors, then
\[
A \rightorientedtimes_C B \coloneq A \times_{C} \Arr(C) \times_{C} B \period
\]
This is the \emph{oriented fiber product} in the bicategory of categories, also sometimes called \enquote{comma construction} for some reason.

Now the \emph{envelope} of $\XX \to \DD$ is the subcategory
\[
  \Env(\XX) \subset \DD \rightorientedtimes_{\DD} \XX
\]
whose objects are those pairs $(\phi \colon \angs{\mu} \to \angs{\lambda}, x) $ in which $\phi$ is dispersive.
We call such an object a \emph{generalized object} of $\XX$. 

The functor $\Env(\XX) \to \DD$ that carries $(\phi \colon \angs{\mu} \to \angs{\lambda}, x)$ to $\mu$ is a cartesian fibration.
The natural functor $\XX \to \Env(\XX)$ is fully faithful;
its left adjoint $\Env(\XX) \to \XX$, which carries $(\phi,x)$ to $x$, is thus a localization.

The envelope $\Env(\DD)$ is the cartesian fibration $\KK \to \DD$ corresponding to the canonical isolability poset $\angs{\lambda} \mapsto K^{\lambda}$. 
Consequently, in general, $\Env(\XX)$ lies over $\KK$.
In particular, when the fibers of $p$ are spaces, $\Env(\XX)$ is an isolability space over the isolability poset $K^{\bullet}$.
The pullback of $\Env(\XX) \to \KK$ over $\DD \to \Env(\DD) = \KK$ recovers $\XX$.

Malthe Sporring pointed out to me that this construction is a special case of a general construction of Haugseng \& Kock \cite{MR4682726} that will work in the presence of any factorization system.

Let $X$ be a Hausdorff topological space.
There is a para-isolability category $\PPi^{\para}(X)$ such that $\Env(\PPi^{\para}(X)) = \Pi(X^{\bullet}/K^{\bullet})$.
An object is a pair $(\angs{\lambda},x)$ consisting of an apartness relation $\angs{\lambda}$ and a map $x \colon V\angs{\lambda} \to X$ that is \emph{strictly separating} in the sense that $(a,b) \in E\angs{\lambda}$ iff $x_a \neq x_b$.
A morphism $(\angs{\lambda},x) \to (\angs{\mu},y)$ is a map $\phi \colon \angs{\lambda} \to \angs{\mu}$ along with a choice, for each $a \in V\angs{\lambda}$, a path $\gamma_a \colon [0,1] \to X$ from $x_a$ to $y_{\phi(a)}$ such that if $t>0$, then $(\phi(a),\phi(b)) \in E\angs{\mu}$ iff $\gamma_a(t) \neq \gamma_b(t)$.

To be more precise, an $n$-simplex of $\PPi^{\para}(X)$ consists of
an $n$-simplex $\Delta^n \to \DD_{\leq 2}$, which is a sequence $\angs{\lambda_0} \to \dots \to \angs{\lambda_n}$ of apartness relations along with
a map $F$ from the realization
\[
  |\Delta^n \times_{\FF} \FF_{\ast}| = \bigcup_{i=0}^n \bigcup_{a \in V\angs{\lambda_i}} |\Delta^{n-i}|
\]
to $X$ satisfying a condition.
To formulate the condition, let us write $F_a$ for the restriction of $F$ to the $(n-i)$-simplex corresponding to a vertex $a \in V\angs{\lambda_i}$ in the union above.
The condition is as follows:
let $a,b \in V\angs{\lambda_i}$ be vertices, let $j \geq i$, and let $t \in |\Delta^{n-i}|$ such that $t_j > 0$;
then $F_a(t) \neq F_b(t)$ iff the images of $a$ and $b$ are connected by an edge in $\angs{\lambda_j}$.

\subsection{$1$-cographs}%
\label{sub:1cographs}
In \cref{sec:combinatorics} we concentrated on undirected graphs, but
we will also consider a directed version of cographs, which we will call \emph{$1$-cographs}.

A \emph{$1$-cograph} $\angs{\gamma} = (V,E)$ consists of a set $V$ of \emph{vertices} along with an antisymmetric, transitive relation $E$ satisfying the following oriented version of the $P_4$-freeness condition \eqref{eqn:P4freeness}:
\begin{equation}\label{eqn:orientedP4freeness}
  \forall w,x,y,z \in V \quad \{(w,x), (y,x), (y,z)\} \subseteq E \implies \{(y,w), (w,z), (z,x)\} \cap E \neq \varnothing \period
\end{equation}
Reflexive $1$-cographs are in particular posets, which are often called \emph{series-parallel posets}.

The theory of $1$-cographs almost perfectly mirrors the theory of cographs.
For example, as in \cref{sub:sums}, there are two ways to combine two $1$-cographs $\angs{\gamma}$ and $\angs{\delta}$:
on one hand, there is the disconnected sum $\angs{\gamma \leftsum \delta}$, in which every element of $\angs{\gamma}$ is incomparable with every element of $\angs{\delta}$;
on the other, there is the ordered connected sum $\angs{\gamma \ordrightsum \delta}$, in which there is an edge $(x,y)$ for every vertex $x$ of $\angs{\gamma}$ and every vertex $y$ of $\angs{\delta}$.

Again there are two singletons -- the irreflexive one $\angs{\ul{1}}$, and the reflexive one $\angs{\overrightarrow{1}}$.
Forming disconnected sums of the irreflexive singleton gives the trivial $1$-cograph $\angs{\ul{n}}$;
forming ordered connected sums of the reflexive singleton gives the totally ordered finite set $\angs{\overrightarrow{n}}$.
The class of $1$-cographs is the smallest class of posets containing the singletons that is closed under these two sums (\cf \cref{sub:depthfiltration}).

For example, by alternating oriented connected and disconnected sums, we can form a $1$-cograph version of $\paw_k$:
\[
  \overrightarrow{\paw}_k \coloneq \angs{(((\ul{1} \ordrightsum \ul{1}) \leftsum \ul{1}) \ordrightsum \ul{1}) \leftsum \cdots \ul{1}} \period
\]

\begin{table}
  \label{table:orientedpawns}
  \centering
  \begin{tabular}{cc}
    $n$ & $\overrightarrow{\paw}_n$ \\
    \hline
    $1$ &
    \begin{tikzpicture}[baseline=0]
      \tikzstyle{vertex}=[circle,fill=black,minimum size=4pt,inner sep=0pt]
      \node[vertex] (v_2) at (0,0) {};
      \node[fit=(current bounding box),inner sep=1mm]{};
    \end{tikzpicture}
    \\
    $2$ &
    \begin{tikzpicture}[baseline=0,decoration={markings,
    mark=at position 0.5 with {\arrow{stealth}}}]
      \tikzstyle{vertex}=[circle,fill=black,minimum size=4pt,inner sep=0pt]
      \node[vertex] (v_1) at (-1,0) {};
      \node[vertex] (v_2) at (0,0) {};
      \draw[postaction=decorate] (v_1) -- (v_2);
      \node[fit=(current bounding box),inner sep=1mm]{};
    \end{tikzpicture}
    \\
    $3$ &
    \begin{tikzpicture}[baseline=0,decoration={markings,
    mark=at position 0.5 with {\arrow{stealth}}}]
      \tikzstyle{vertex}=[circle,fill=black,minimum size=4pt,inner sep=0pt]
      \node[vertex] (v_1) at (-1,0) {};
      \node[vertex] (v_2) at (0,0) {};
      \node[vertex] (v_3) at (1,0) {};
      \draw[postaction=decorate] (v_1) -- (v_2);
      \node[fit=(current bounding box),inner sep=1mm]{};
    \end{tikzpicture}
    \\
    $4$ &
    \begin{tikzpicture}[baseline=0,decoration={markings,
    mark=at position 0.5 with {\arrow{stealth}}}]
      \tikzstyle{vertex}=[circle,fill=black,minimum size=4pt,inner sep=0pt]
      \node[vertex] (v_1) at (-1,-0.5) {};
      \node[vertex] (v_2) at (-1,0.5) {};
      \node[vertex] (v_3) at (0,0)  {};
      \node[vertex] (v_4) at (1,0)  {};
      \draw[postaction=decorate] (v_1) -- (v_2);
      \draw[postaction=decorate] (v_1) -- (v_3);
      \draw[postaction=decorate] (v_2) -- (v_3);
      \draw[postaction=decorate] (v_4) -- (v_3);
      \node[fit=(current bounding box),inner sep=1mm]{};
    \end{tikzpicture}
    \\
    $5$ &
    \begin{tikzpicture}[baseline=0,decoration={markings,
    mark=at position 0.5 with {\arrow{stealth}}}]
      \tikzstyle{vertex}=[circle,fill=black,minimum size=4pt,inner sep=0pt]
      \node[vertex] (v_1) at (-1,-0.5) {};
      \node[vertex] (v_2) at (-1,0.5) {};
      \node[vertex] (v_3) at (0,0) {};
      \node[vertex] (v_4) at (1,0) {};
      \node[vertex] (v_5) at (2,0) {};
      \draw[postaction=decorate] (v_1) -- (v_2);
      \draw[postaction=decorate] (v_1) -- (v_3);
      \draw[postaction=decorate] (v_2) -- (v_3);
      \draw[postaction=decorate] (v_4) -- (v_3);
      \node[fit=(current bounding box),inner sep=1mm]{};
    \end{tikzpicture}
    \\
    $6$ &
    \begin{tikzpicture}[baseline=0,decoration={markings,
    mark=at position 0.5 with {\arrow{stealth}}}]
      \tikzstyle{vertex}=[circle,fill=black,minimum size=4pt,inner sep=0pt]
      \node[vertex] (v_1) at (-1.155,0) {};
      \node[vertex] (v_2) at (-0.577,1) {};
      \node[vertex] (v_3) at (-0.577,0.333) {};
      \node[vertex] (v_4) at (-0.577,-1) {};
      \node[vertex] (v_5) at (0,0) {};
      \node[vertex] (v_6) at (1,0) {};
      \draw[postaction=decorate] (v_2) -- (v_1);
      \draw[postaction=decorate] (v_3) -- (v_1);
      \draw[postaction=decorate] (v_4) -- (v_1);
      \draw[postaction=decorate] (v_3) -- (v_2);
      \draw[postaction=decorate] (v_1) -- (v_5);
      \draw[postaction=decorate] (v_2) -- (v_5);
      \draw[postaction=decorate] (v_3) -- (v_5);
      \draw[postaction=decorate] (v_4) -- (v_5);
      \draw[postaction=decorate] (v_6) -- (v_5);
      \node[fit=(current bounding box),inner sep=1mm]{};
    \end{tikzpicture}
    \\
  \end{tabular}
  \caption{The first few $1$-cographs $\overrightarrow{\paw}_n$}
\end{table}

We write $\GG^{[1]}$ for the category of $1$-cographs and relation-preserving maps.
Now $(\GG^{[1]},\leftsum,\ordrightsum)$ is a twofold monoidal $1$-category in which the monoidal structure on the left is symmetric.
For short, we call these $\infty \choose 1$-monoidal $1$-categories.
Contained in $\GG^{[1]}$ are the $\infty \choose 1$-monoidal full subcategories $\DD^{[1]}$ of irreflexive $1$-cographs (which in particular are strict partial orders) and $\EE^{[1]}$ of reflexive $1$-cographs (which in particular are posets).
The vertex functor $V \colon \GG^{[1]} \to \FF$ is a cartesian fibration,
and $\DD^{[1]}$ and $\EE^{[1]}$ are full fibered subcategories of $\GG^{[1]}$.

The category $\FF$ can be identified with the full subcategory $\DD^{[1]}_{\leq 1} \subset \DD^{[1]}$ consisting of the trivial $1$-cographs $\angs{\ul{n}}$.
The category $\OO$ of totally ordered finite sets can be identified with the full subcategory $\EE^{[1],\leq 1} \subset \EE^{[1]}$ consisting of the totally ordered finite sets $\angs{\overrightarrow{n}}$.

Alternating the two sums, we construct exhaustive filtrations
\[
  \GG^{[1]}_{\leq \bullet} \subset \GG^{[1]} \comma \qquad
  \GG^{[1],\leq \bullet} \subset \GG^{[1]} \comma \qquad
  \DD^{[1]}_{\leq \bullet} \subset \DD^{[1]} \comma \andeq
  \EE^{[1],\leq \bullet} \subset \EE^{[1]}
\]
by full subcategories fibered over $\FF$ (\cf \cref{sub:depthfiltration}).
For example, the objects of $\DD^{[1]}_{\leq 2}$ are \emph{oriented apartness relations},
which take the form $\angs{\ul{n}_1 \ordrightsum \cdots \ordrightsum \ul{n}_k}$.

There are some contrasts with $\DD$ and $\EE$.
Most importantly, observe that $\DD^{[1]}$ and $\EE^{[1]}$ are not vertical opposites.
Rather, to get an explicit model of $\DD^{[1],\vop}$, one can take a category of bicolored complete graphs where one color (say \enquote{red}) forms a reflexive cograph and the other (say \enquote{blue}) forms an irreflexive $1$-cograph.
This observation and various generalizations are due to Willow Bevington and Malthe Sporring.

A new thing we can do with $1$-cographs is to form the \emph{opposite} $\angs{\gamma}^{\op}$ of a $1$-cograph $\angs{\gamma}$.
This operation defines an automorphism $\GG^{[1]} \to \GG^{[1]}$, which preserves the two singletons, is compatible with $\leftsum$, and reverses the order of $\ordrightsum$, in the sense that $\angs{\gamma \ordrightsum \delta}^{\op} = \angs{\delta}^{\op} \ordrightsum \angs{\gamma}^{\op}$.

We have a forgetful functor $S \colon \GG^{[1]} \to \GG$.
If $\angs{\gamma}$ is a $1$-cograph with $\angs{n} = V\angs{\gamma}$, then
\[
  S\angs{\gamma} = \angs{\gamma} \cup^{\angs{\ul{n}}} \angs{\gamma}^{\op} \period
\]
Note that this forgetful functor is symmetric monoidal with respect to $\leftsum$ and monoidal with respect to $\ordrightsum$.
If $\angs{\lambda} = S\angs{\gamma}$, then we say that $\angs{\gamma}$ is a \emph{$1$-structure} on $\angs{\lambda}$.
The collection of all $1$-structures over $\angs{\lambda}$ -- \emph{i.e.}, the fiber $\GG^{[1]}_{\angs{\lambda}}$ of $S$ over $\angs{\lambda}$ -- is a set.

The forgetful functor $S$ is not quite a right fibration.
In other words, the assignment that carries a cograph $\angs{\lambda}$ to its set $\GG^{[1]}_{\angs{\lambda}}$ of $1$-structures isn't a functor.
For example, the transitivity property can be lost when we pass to an undirected sub-cograph of a $1$-cograph.
For example, $\angs{\ul{1} \rightsum \ul{2}}$ is contained in $S\angs{\ul{1} \ordrightsum \ul{1} \ordrightsum \ul{1}}$:
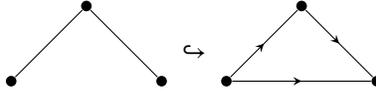
\[
  \begin{tikzpicture}[baseline=8,decoration={markings,
  mark=at position 0.5 with {\arrow{stealth}}}]
    \tikzstyle{vertex}=[circle,fill=black,minimum size=4pt,inner sep=0pt]
    \node[vertex] (v_1) at (-1,0) {};
    \node[vertex] (v_2) at (0,1) {};
    \node[vertex] (v_3) at (1,0) {};
    \draw (v_1) -- (v_2);
    \draw (v_2) -- (v_3);
    \node[fit=(current bounding box),inner sep=1mm]{};
  \end{tikzpicture}
  \inclusion
  \begin{tikzpicture}[baseline=8,decoration={markings,
mark=at position 0.5 with {\arrow{stealth}}}]
    \tikzstyle{vertex}=[circle,fill=black,minimum size=4pt,inner sep=0pt]
    \node[vertex] (v_1) at (-1,0) {};
    \node[vertex] (v_2) at (0,1) {};
    \node[vertex] (v_3) at (1,0) {};
    \draw[postaction=decorate] (v_1) -- (v_2);
    \draw[postaction=decorate] (v_1) -- (v_3);
    \draw[postaction=decorate] (v_2) -- (v_3);
    \node[fit=(current bounding box),inner sep=1mm]{};
  \end{tikzpicture}
\]

In effect, this is the only thing that can go wrong.
The functor $S$ \emph{does} restrict to a right fibration over the accretive morphisms.
More precisely, for every accretive morphism $f \colon \angs{\lambda} \to \angs{\mu}$ of $\GG$ and every $1$-structure $\angs{\delta}$ on $\angs{\mu}$,
there is an $S$-cartesian lift $\angs{\gamma} \to \angs{\delta}$ of $f$.
In particular, $\DD^{[1]}$ is a para-isolability space in the sense of \cref{sub:envelopes}.

One can therefore form the envelope of $\DD^{[1]}$ to obtain the isolability stratification 
\[
  \Env(\DD^{[1]}) \to \DD \comma
\]
which lies over the canonical isolability poset $K^{\bullet}$.

More explicitly, the objects of $\Env(\DD^{[1]})$ are \emph{generalized $1$-cographs}, which are pairs $(\angs{\lambda},\angs{\gamma})$ consisting of
a $1$-cograph $\angs{\gamma}$ and
a subcograph $\angs{\lambda} \subseteq S\angs{\gamma}$ with the same vertices.

The $2$-skeleton of the envelope $\Env(\DD^{[1]})$ turns out to be equivalent to the isolability stratification attached to the real line, as we shall now show.

\subsection{The isolability line \& $1$-cographs}%
\label{sub:isolabilityline}
The isolability space $\Pi(\RR^\bullet/K^\bullet)$ is what we call the \emph{isolability line}.
The value on a cograph $\angs{\lambda}$ is the stratified space of separating functions $x \colon V\angs{\lambda} \to \RR$.
Here we present a more combinatorial description of the isolability line.

Recall (\cref{sub:1cographs}) that irreflexive $1$-cographs form a para-isolability space $\DD^{[1]}$;
we will actually only be interested in the segment $\DD^{[1]}_{\leq 2}$ consisting of the oriented apartness relations.
The envelope (\cref{sub:envelopes}) of $\DD^{[1]}_{\leq 2}$ is the isolability space $L^{\bullet}$ in which each $\angs{\lambda}$ is carried to $L^{\lambda} \coloneq K^{\lambda} \times_{\DD} \DD^{[1]}_{\leq 2}$.
In other words, $L^\lambda$ is the poset of dispersive maps $j \colon \angs{\lambda} \to S\angs{\gamma}$, where $\angs{\gamma}$ is an oriented apartness relation.

The topological space $\RR^\lambda$ is naturally stratified over the poset $L^\lambda$:
to every separating function $x \colon V\angs{\lambda} \to \RR$ we assign the dispersive map $\angs{\lambda} \to S\angs{\gamma_x}$, where $\angs{\gamma_x}$ is the $1$-cograph with the same vertices as $\angs{\lambda}$, but now $(a,b) \in E\angs{\gamma_x}$ iff $x_a < x_b$.
Furthermore, the strata and links of the $E^{\lambda}$-stratified topological space $\RR^{\lambda}$ are all contractible.
Consequently:
\begin{prp}
  $\Pi(\RR^{\bullet}/K^{\bullet}) = L^{\bullet}$.
\end{prp}

Using the tensor product of \cref{sub:products}, we can stratify the isolability topological space $(\RR^n)^{\bullet}$ over the isolability poset $(K^{\otimes n})^{\bullet} \coloneq (K \otimes \cdots \otimes K)^{\bullet}$.
We obtain:
\begin{cor}
  $\Pi((\RR^n)^{\bullet}/(K^{\otimes n})^{\bullet}) = (L^{\otimes n})^{\bullet}$.
\end{cor}

By changing stratification along $K^{\otimes n} \to K$,
this becomes a description $\Pi((\RR^n)^{\bullet}/K^{\bullet})$ in combinatorial terms.
(There is a more direct way to think about this isolability space,
which we will explore in future work with Malthe Sporring.)

\subsection{Ran spaces in the homotopical context}%
\label{sub:ran}
Recall that a fibered category $u \colon \AA \to \FF$ gives a factorization system $(\AA^{\disp},\AA_{\accr})$, where $\AA^{\disp}$ consists of the dispersive -- \ie, $u$-inverted -- maps, and $\AA_{\accr}$ consists of the accretive -- \ie, $u$-cartesian -- maps.
Using this factorization system, in \cref{sub:spans} we introduced the notion of a \emph{vertical opposite} $A^{\vop}$ of a fibered category $\AA \to \FF$ as the category of spans in $\AA$ where the backward maps are dispersive and the forward maps are accretive.
The \emph{horizontal opposite} is now the span category
\[
  \AA^{\hop} \coloneq \AA^{\vop,\op} = \Span(\AA;\AA^{\disp},\AA_{\accr}) 
\]
where the forward maps are dispersive and the backward maps are accretive.

For example, $\DD^{\hop} = \EE^{\op}$ (via negation).
Consider the section $\FF \inclusion \EE$ that carries $\angs{n}$ to $\angs{\ul{n}}_{\textit{refl}} = \angs{\ul{n} \cdot \ol{1}} = \angs{\ol{1} \leftsum \cdots \leftsum \ol{1}}$;
if we apply horizontal opposites, this becomes $\FF^{\op} \inclusion \DD^{\hop}$, which carries $\angs{n}$ to $\angs{\ol{n}}_{\textit{irr}} = \angs{\ol{n} \cdot \ul{1}} = \angs{\ul{1} \rightsum \cdots \rightsum \ul{1}}$.

Let us consider an isolability space $X^\bullet$ presented as a cartesian fibration $\XX \to \DD$.
By composition we regard $\XX$ as fibered over $\FF$.
Now the \emph{unital Ran space} of $X^{\bullet}$ is the category
\[
  \Ran^u(X^\bullet) \coloneq \XX^{\hop} \times_{\DD^{\hop}} \FF^{\op}
\]
over $\FF^{\op}$.
The \emph{Ran space} of $X^{\bullet}$ is the further pullback
\[
  \Ran(X^\bullet) \coloneq \XX^{\hop} \times_{\DD^\hop} \FF_s^{\op} \period
\]

If $\angs{\lambda}$ is an irreflexive cograph that admits an accretive map $\angs{\lambda} \to \angs{\ol{n} \cdot \ul{1}}$, then
necessarily $\angs{\lambda} \in \DD_{\leq 2}$.
That means that the section $\FF^{\op} \inclusion \DD^{\hop}$ factors through $\DD^{\hop}_{\leq 2} \subset \DD^{\hop}$.
Consequently, the Ran spaces attached to $X^{\bullet}$ only depend upon the $2$-skeleton $\sk_2(X^\bullet)$ (\cref{sub:skeletality}).

Let us consider the special case of $\Ran(\Pi(X^{\bullet}/K^{\bullet}))$. 
An object $(\angs{n},x)$ consists of a finite set $\angs{n}$ along with a point $x \in X^{\ol{n} \cdot \ul{1}}$.
This $x$ is a configuration of $n$ distinct points of $X$.
A map $(\angs{n},x) \to (\angs{m},y)$ is now a map $f \colon \angs{m} \to \angs{n}$ of finite sets along with a path from $x_j$ to $y_i$ whenever $f(i) = j$ which meet only at the $x_j$'s. 

More specifically, if $X$ is a manifold, then this Ran space is equivalent to the exit-path category of the unital Ran space as constructed by Cepek \cite[Df. 1.0.4]{cepekcombinatorics}, and
the Ran space $\Ran(\Pi(X^{\bullet}/K^\bullet))$ is the exit-path category of the usual (nonunital) Ran space \cite[Df. 1.0.10]{cepekcombinatorics}.

Our combinatorial descriptions of the isolability line and isolability $n$-space (\cref{sub:isolabilityline}) now reproduce a variant of the Cepek's theorem:
\begin{prp}
  $\Ran^u(\Pi((\RR^n)^\bullet/(K^{\otimes n})^\bullet)) = \OO^n$,
  where $\OO^n$ is the $n$-fold product of the category $\OO$ of totally ordered finite sets.
\end{prp}
By changing the stratification along $(K^{\otimes n})^{\bullet} \to K^{\bullet}$,
we recover Cepek's identification \cite[Th. 1.2.1]{cepekcombinatorics}:
\[
  \Ran^u(\Pi((\RR^n)^\bullet/K^\bullet)) = \OO^{(n)} \comma
\]
where $\OO^{(n)}$ is the free $n$-monoidal category on an $n$-monoid --
or, equivalently, the $n$-fold wreath product of $\OO$ with itself, with bijections inverted \cite{MR3784514}.


\section{Twofold symmetric monoidal structures}%
\label{sec:twofold}
Let $\angs{\lambda} \mapsto X^{\lambda}$ be an isolability space.
Let us imagine that we are interested in a theory of \enquote{sheaves}, which are organized into a diagram of symmetric monoidal categories $T \mapsto \AA(T)$.
What structure does the isolability structure on $X$ induce on the symmetric monoidal category $\AA(X)$?

The easy answer is that we have a diagram $\angs{\lambda} \mapsto \AA(X^{\lambda})$ of symmetric monoidal categories.
This structure arises from \emph{right external} tensor products
\[
  \rightboxtensor \colon \AA(X^{\lambda}) \times \AA(X^{\mu}) \to \AA(X^{\lambda \leftsum \mu}) \semicolon
\]
the internal tensor product on $\AA(X^{\lambda})$ arises from restricting the right external product along the diagonal $X^{\lambda} \to X^{\lambda \leftsum \lambda}$.

But this structure smuggles a different structure into our story:
by restricting along our maps $X^{\lambda \rightsum \mu} \to X^{\lambda \leftsum \mu}$, we obtain \emph{left} external tensor products
\[
  \leftboxtensor \colon \AA(X^{\lambda}) \times \AA(X^{\mu}) \to \AA(X^{\lambda \rightsum \mu}) \period
\]
Factorization algebras are defined with respect to these left external tensor products.

In order to explain how these two external products interact, we will understand them as a \enquote{twofold symmetric monoidal structure} on the diagram $\angs{\lambda} \mapsto \AA(X^{\lambda})$.

Roughly speaking, a twofold symmetric monoidal structure on a category $C$ is a pair of symmetric monoidal structures $\lefttensor,\righttensor$ that are required to share a unit and to have \emph{intertwiner maps}
\[
  (U \righttensor V) \lefttensor (X \righttensor Y) \to (U \lefttensor X) \righttensor (V \lefttensor Y) \comma
\]
natural in $U$, $V$, $X$, and $Y$.
In the $1$-categorical context, these were defined by Balteanu, Fiedorowicz, Schwänzl, and Vogt \cite{MR1982884}.
These are special cases of \emph{duoidal} $1$-categories \cite{MR3040601}, in which the units of the monoidal structures are not required to coincide.
Takeshi Torii \cite{MR4899410} described the fully homotopical theory of duoidal categories.

For our purposes, it's technically more convenient to give a general, direct definition.
We emphasize that we are just outlining some aspects of the theory of \emph{twofold algebra} here;
a more complete development will appear elsewhere.

\subsection{Fibrations of cographs \& indexed sums}%
\label{sub:fibrations}
Before we discuss the general theory of twofold symmetric monoidal structures, let us consider how the category of cographs itself provides an example.

Let $\angs{\lambda}$ be a reflexive cograph, and
for each vertex $a \in V\angs{\lambda}$, let $\angs{\mu_a}$ be a cograph.
We can define an \emph{indexed sum}
\[
  \angs{\mu} = \bigoplus_{a \in \angs{\lambda}} \angs{\mu_a}
\]
inductively by cograph depth:
if $\angs{\lambda} = \angs{\ol{1}}$, then we let $\angs{\mu} = \angs{\mu_1}$;
if $\angs{\lambda} = \angs{\lambda_1 \leftsum \lambda_2}$, then we let
\[
  \angs{\mu} = \bigoplus_{a \in \angs{\lambda_1}} \angs{\mu_a} \leftsum \bigoplus_{b \in \angs{\lambda_2}} \angs{\mu_b} \semicolon
\]
and if $\angs{\lambda} = \angs{\lambda_1 \rightsum \lambda_2}$, then we let
\[
  \angs{\mu} = \bigoplus_{a \in \angs{\lambda_1}} \angs{\mu_a} \rightsum \bigoplus_{b \in \angs{\lambda_2}} \angs{\mu_b} \period
\]

Thus the two distinct sums $\leftsum$ and $\rightsum$ on cographs can be seen as special cases of this single indexed sum.
This is the model for our definition of twofold symmetric monoidal categories in general:
rather than specifying two distinct symmetric monoidal structures and their compatibilities, we instead codify twofold structures in terms of a single indexed direct sum.

In the category of finite sets, a map $\angs{m} \to \angs{n}$ exhibits $\angs{m}$ as the coproduct
\[
  \angs{m} = \coprod_{i \in \angs{n}} \angs{m_i}
\]
of the various fibers.
The analogous statement for indexed sums of cographs isn't quite correct
(think of the dispersive map $\angs{\ol{1} \leftsum \ol{1}} \to \angs{\ol{2}}$),
but it's not difficult to specify the maps for which this is true.

Call a map $\angs{\mu} \to \angs{\lambda}$ of cographs a \emph{fibration} iff the pullback to the maximal irreflexive subcograph
\[
  \angs{\mu} \times_{\angs{\lambda}} \angs{\lambda}_{\textit{irr}} \to \angs{\lambda}_{\textit{irr}}
\]
is accretive.
If $\angs{\lambda}$ is reflexive, then every fibration $\angs{\mu} \to \angs{\lambda}$ exhibits $\angs{\mu}$ as the indexed sum 
\[
  \angs{\mu} = \bigoplus_{a \in \angs{\lambda}} \angs{\mu_a} \period
\]

\subsection{Partial maps of cographs}%
\label{sub:LambdaGG}
In order to define the indexing categories for twofold symmetric monoidal structures, we will need to enlarge the category of cographs by including maps of cographs that are not defined everywhere.
The analogy here is with the passage from the category $\FF$ of finite sets to the category $\Lambda(\FF)$ of pointed finite sets --
or, equivalently, finite sets and partially-defined maps.

Let $\Lambda(\GG)$ be the category in which an object is a cograph, and
a morphism is a \emph{partial map} $\angs{\lambda} \partto \angs{\mu}$ of cographs that is defined on an induced subgraph $\angs{\lambda'} \subseteq \angs{\lambda}$.
That is, if $\GG^{\dag} \subset \GG$ is the wide subcategory consisting of accretive injections of cographs, then
$\Lambda(\GG)$ is the span category $\Span(\GG;\GG,\GG^{\dag})$.
Equivalently, $\Lambda(\GG)$ is the Kleisli category of the monad on $\GG$ given by $\angs{\mu} \mapsto \angs{\mu \rightsum \overline{1}}$.
(It is an example of the \enquote{Leinster category} of a perfect operator category, in the sense of \cite{MR3784514}.)

In the same way, we define wide subcategories $\DD^{\dag}$ and $\EE^{\dag}$ of accretive injections,
and we have categories of partial maps
\[
  \Lambda(\DD) = \Span(\DD;\DD,\DD^{\dag}) \andeq
  \Lambda(\EE) = \Span(\EE;\EE,\EE^{\dag}) \period
\]
These are full subcategories of $\Lambda(\GG)$.

We call a partial map of cographs \emph{inert} if it lies in $\GG^{\dag,\op} \subset \Lambda(\GG)$.
For example, given a reflexive cograph $\angs{\lambda}$ and a vertex $a \in V\angs{\lambda}$, we obtain an inert partial map $\chi_a \colon \angs{\lambda} \to \angs{\ol{1}}$;
these maps appear in our definition of twofold commutative structures in the next subsection.
We call a partial map \emph{active} if it lies in $\GG \subset \Lambda(\GG)$.
These give an orthogonal factorization system on $\Lambda(\GG)$.

The categories $\Lambda(\DD)$, $\Lambda(\EE)$, and $\Lambda(\GG)$ all lie over the category $\Lambda(\FF)$ of finite sets and partial maps.
Various natural functors over $\FF$ among $\DD$, $\EE$, $\GG$, and $\FF$ induce functors among their \enquote{partial} variants.

For example, we have two fully faithful functors $\FF \inclusion \EE$.
The first is the inclusion $\FF = \DD_{\leq 1} \inclusion \GG$ followed by the formation of the reflexive hull:
$\angs{n} \mapsto \angs{\underline{n}}_{\textit{refl}}$.
The second is the inclusion $\FF = \EE^{\leq 1} \inclusion \EE$, given by the assignment $\angs{n} \mapsto \angs{\overline{n}}$.
These extend to fully faithful left and right adjoints of the functor $\Lambda(\EE) \to \Lambda(\FF)$.
These two functors will appear below (\cref{sub:Twofold});
they are how one extracts the underlying symmetric monoidal structures from twofold symmetric monoidal structures in our framework below.

\subsection{Twofold symmetric monoidal bicategories}%
\label{sub:Twofold}
To deal with the coherences that arise, a small amount of \emph{bicategory} theory appears.%
\footnote{In keeping with our conventions, by \enquote{bicategory} we mean \enquote{$(\infty,2)$-category}.}
We follow terminology and constructions from the following references: \cite{MR4305242,MR4498545,MR4519636,MR4845972,abellan}.
Here are some examples of bicategories that will be relevant for us:
\begin{itemize}
  \item[$\Cat$,] the bicageory of categories,
  \item[$\biCat$,] the bicategory of bicategories,
  \item[$\Pr^L$,] the bicategory presentable categories,
  \item[$\Mod(\AA)$,] the bicategory of $\AA$-modules for a monoidal presentable category $\AA$, relative to the Lurie tensor product.
\end{itemize}

The basic strategy is to define twofold commutative structures first relative to cartesian products, and then use this to define them more generally.
So, for a bicategory $\MM$ with finite products, such as $\Cat$ or $\biCat$,
we define \enquote{twofold commutative monoids} in $\MM$.
This gives us access to \enquote{twofold symmetric monoidal categories} and \enquote{twofold symmetric monoidal bicategories}.
Every symmetric monoidal (bi)category is automatically a twofold symmetric monoidal (bi)category in which the two tensor products coincide.
Next, given a twofold symmetric monoidal bicategory $\CC^{\lefttensor,\righttensor}$, we will define a \enquote{twofold commutative monoid}.
For instance, a \enquote{twofold symmetric monoidal presentable category} is a twofold commutative monoid in $\Pr^{L,\otimes,\otimes}$.

Let $\MM$ be a bicategory with finite products.
A twofold commutative monoid of $\MM$ will be a normalized oplax functor 
\[
  C^{\lefttensor,\righttensor} \colon \Lambda(\EE) \to \MM
\]
satisfying certain conditions.
We will write $C$ for the value of this normalized oplax functor on the singleton $\angs{\overline{1}}$.
Note that $C^{\lefttensor,\righttensor}$ restricts, via the two inclusions $i,j \colon \Lambda(\FF) \inclusion \Lambda(\EE)$, to two normalized oplax functors
\[
  C^{\lefttensor} \coloneq i^{\ast} C^{\lefttensor,\righttensor} \colon \Lambda(\FF) \to \MM \andeq
  C^{\righttensor} \coloneq j^{\ast} C^{\lefttensor,\righttensor} \colon \Lambda(\FF) \to \MM \comma
\]
which are connected by a natural transformation $C^{\lefttensor} \to C^{\righttensor}$.

There will be two conditions we impose on our normalized oplax functor $C^{\lefttensor,\righttensor}$;
one of these is the usual Segal condition, but
the other is more subtle.
It is the requirement that $C^{\lefttensor,\righttensor}$ carry certain $2$-simplices in $\Lambda(\EE)$ to invertible $2$-morphisms in $\MM$.
We now need to identify those $2$-simplices.

Let $\sigma$ be a $2$-simplex of $\Lambda(\EE)$ consisting of a partial map $\angs{\lambda} \supseteq \angs{\lambda'} \to \angs{\mu}$, a partial map $\angs{\mu} \supseteq \angs{\mu'} \partto \angs{\nu}$, and their composite $\angs{\lambda} \supseteq \angs{\lambda''} \to \angs{\nu}$.
We now say that $\sigma$ is \emph{thin} iff the induced map $\angs{\lambda''} \to \angs{\mu'}$ is a fibration.

Now our normalized oplax functor $C^{\lefttensor,\righttensor}$ is a \emph{twofold commutative monoid of $\MM$} iff it satisfies the following conditions:
\begin{enumerate}
  \item Every thin $2$-simplex of $\Lambda(\EE)$ is carried to an invertible $2$-morphism in $\MM$.
 \item For every reflexive cograph $\angs{\lambda}$, the Segal map
    \[
      C^{\lefttensor,\righttensor}_{\lambda} \to \prod_{a \in V\angs{\lambda}} C = C^{V\angs{\lambda}}
    \]
    induced by the inert partial maps $\chi_a \colon \angs{\lambda} \partto \angs{\overline{1}}$ is an equivalence.
\end{enumerate}

The objects $C^{\lefttensor}$ and $C^{\righttensor}$ are then commutative monoid objects of $\MM$,
which we call the \emph{left commutative monoid structure} and the \emph{right commutative monoid structure} on $C$.
These structures come with an \emph{intertwiner} natural transformation
\[
  \nu \colon (U \righttensor V) \lefttensor (X \righttensor Y) \to (U \lefttensor X) \righttensor (V \lefttensor Y) \period
\]
induced by the active map of cographs 
\[
  \angs{(\overline{1} \rightsum \overline{1}) \leftsum (\overline{1} \rightsum \overline{1})} \to \angs{(\overline{1} \leftsum \overline{1}) \rightsum (\overline{1} \leftsum \overline{1})} \period
\]
They also share a unit, which is given by the functor $\{1\} = C^{\lefttensor,\righttensor}_{0} \to C$ induced by the unique active map from the empty cograph to the singleton.
We thus also have a map of commutative monoids $C^{\lefttensor} \to C^{\righttensor}$.
In the same manner, all the coherences of the theory are extracted from diagrams in the category $\EE$.

\subsection{Examples}%
\label{sub:ExamplesTwofold}
We call twofold commutative monoids in the bicategory $\MM = \Cat$ \emph{twofold symmetric monoidal categories}.
In this case, it is helpful to recast the normalized oplax functor $C^{\lefttensor,\righttensor}$ as a locally cocartesian fibration
\[
  \CC^{\lefttensor,\righttensor} \to \Lambda(\EE) \period
\]
Then our conditions become:
\begin{enumerate}
  \item In every square
    \[
      \begin{tikzcd}[sep=1.5em, ampersand replacement=\&]
        \Delta^{\{0,1\}} \arrow[r, "f"] \arrow[d] \& \CC^{\lefttensor,\righttensor} \arrow[d] \\
        \Delta^2 \arrow[r, "\sigma"']  \& \Lambda(\EE) 
      \end{tikzcd}
    \]
    in which $\sigma$ is thin, the morphism $f$ is cocartesian in $\CC^{\lefttensor,\righttensor} \times_{\Lambda(\EE)} \Delta^2$.
  \item The Segal maps $\CC^{\lefttensor,\righttensor}_{\lambda} \to \CC^{u\angs{\lambda}}$ are equivalences.
\end{enumerate}

Similarly, \emph{twofold symmetric monoidal bicategories} are twofold commutative monoid objects of $\biCat$.
Thanks to the straightening/unstraightening equivalences of Fernando Abellán \cite{abellan}, these can be presented as local $(0,1)$-fibrations over $\Lambda(\EE)$ satisfying the three conditions above.

If $A^{\otimes}$ is a symmetric monoidal category, we can pull it back along the forgetful functor $u \colon \Lambda(\EE) \to \Lambda(\FF)$ to obtain a twofold symmetric monoidal category $A^{\otimes,\otimes}$.

A twofold symmetric monoidal $1$-category in the sense of \cite{MR1982884} is automatically a twofold symmetric monoidal category.
The categories $\GG$, $\DD$, and $\EE$ are thus all examples.

Let's describe $\GG^{\leftsum, \rightsum} \to \Lambda(\EE)$ explicitly.
Let $\Fib(\GG)$ be the full subcategory of $\Arr(\GG)$ consisting of fibrations
$\angs{\nu} \fibration \angs{\lambda}$
with $\angs{\lambda}$ reflexive.
Now
\[
  \GG^{\leftsum,\rightsum} = \Span(\Fib(\GG), \Fib(\GG), \Fib(\GG)^{\dag}) \comma
\]
where $\Fib(\GG)^{\dag}$ consists of pullback squares
\[
  \begin{tikzcd}[sep=1.5em, ampersand replacement=\&]
    \angs{\nu'} \arrow[r, "f'"] \arrow[d, ->>] \& \angs{\nu} \arrow[d, ->>] \\
    \angs{\lambda'} \arrow[r, "f"']  \& \angs{\lambda} 
  \end{tikzcd}
\]
in which $f$ and (thus) $f'$ are accretive injections.
The locally cocartesian fibration $\GG^{\leftsum,\rightsum} \to \Lambda(\EE)$ is induced by the target functor $\Fib(\GG) \to \EE$.
Now $\DD^{\leftsum,\rightsum}$ and $\EE^{\leftsum,\rightsum}$ are the full subcategories spanned by those fibrations 
$\angs{\nu} \fibration \angs{\lambda}$
in which the source $\nu$ is irreflexive and reflexive, respectively.

\subsection{Twofold monoids in twofold symmetric monoidal categories}%
\label{sub:TwofoldinTwofold}
Let $\CC^{\lefttensor,\righttensor} \to \Lambda(\EE)$ be a twofold symmetric monoidal bicategory, presented as a local $(0,1)$-fibration.
A \emph{twofold commutative monoid} in $\CC^{\lefttensor,\righttensor}$ is a section $\Lambda(\EE) \to \CC^{\lefttensor,\righttensor}$ that carries inert partial maps to cocartesian morphisms.

A twofold commutative monoid in $\CC^{\lefttensor,\righttensor}$ is thus an object $X = X_1 \in C$ along with two commutative monoid structures $\lefttensor \colon X \lefttensor X \to X$ and $\righttensor \colon X \righttensor X \to X$ with a common unit $\iota \colon \mathbf{1}_C \to X$.
These two structures are compatible via the intertwiner map $\nu$. 

We write $\CAlg^{(2)}(\CC^{\lefttensor,\righttensor})$ for the bicategory of twofold commutative monoids in $\CC^{\lefttensor, \righttensor}$.

If $\MM$ is a bicategory with finite products, then one shows in the usual way that a twofold commutative monoid in the twofold symmetric monoidal category $\MM^{\times,\times}$ is essentially the same thing as a twofold commutative monoid in $\MM$ in the sense of \cref{sub:Twofold}.

\subsection{Twofold Day convolution}%
\label{sub:TwofoldDay}
Let $A^{\lefttensor,\righttensor}$ be a twofold symmetric monoidal category.
On $\Fun(A,\Cat)$, there is a twofold symmetric monoidal structure, in which each of the two symmetric monoidal structures $\lefttensor$ and $\righttensor$ is obtained by Day convolution with respect to $\righttensor$ and $\lefttensor$ (respectively -- note the reversal!) on $A$.
To construct this two fold symmetric monoidal structure precisely, we follow a clever strategy of Hadrian Heine \cite[\S 6.1]{heinethesis}.

Consider the bicategory $\categ{Cocart}$ of cocartesian fibrations of categories.
This is symmetric monoidal under product, and there is a symmetric monoidal functor $\categ{Cocart}^{\times} \to \Cat^{\times}$ given by passage to the target.
We can then pull this back to a twofold symmetric monoidal functor $\categ{Cocart}^{\times,\times} \to \Cat^{\times,\times}$, which is a local $(0,1)$-fibration.
The two fold symmetric monoidal category $A^{\lefttensor,\righttensor}$ can then be expressed as a section $\Lambda(\EE) \to \Cat^{\times,\times}$.
Pulling back $\categ{Cocart}^{\times,\times} \to \Cat^{\times,\times}$ along this section, we obtain a local $(0,1)$-fibration
\[
  \Fun(A,\Cat)^{\lefttensor,\righttensor} \coloneq
  \Lambda(\EE) \times_{\Cat^{\times,\times}} \categ{Cocart}^{\times,\times} \to \Lambda(\EE) \comma
\]
which is a twofold symmetric monoidal bicategory.

It follows from Heine's work that the two symmetric monoidal structures $\lefttensor$ and $\righttensor$ on $\Fun(A,\Cat)$ are given by Day convolutions.

\subsection{Comments \& questions}%
\label{sub:comments-twofold}
Here are some basic statements: 
\begin{itemize}
  \item $\EE$ is the free twofold symmetric monoidal category supporting a twofold commutative monoid.
  \item $\DD$ is the free twofold symmetric monoidal category supporting a commutative monoid for the left tensor product $\leftsum$.
  \item $\DD^{\disp}$, which is equivalent to $\EE^{\disp}$ (via the functor $\angs{\lambda} \mapsto \angs{\lambda}_{\textit{refl}}$) and also to $\DD^{\disp,\op}$ and $\EE^{\disp,\op}$ (via negation), is the free twofold symmetric monoidal category generated by a single object.
    (This observation is due to Malthe Sporring.)
\end{itemize}
These statements can be certainly proved with the definitions we give here, but that isn't surprising:
our definitions more or less bake these assertions into the theory.

A more satisfying account of these structures will develop a theory even of \enquote{$m \choose n$-monoidal categories} --
\ie, $E_m$-monoids in the $2$-category $\categ{Mon}^{(n),\textit{oplax}}\categ{Cat}$ of $n$-monoidal categories and normal oplax $n$-monoidal functors --
from first principles, without direct reliance on the combinatorics of cographs.
For $1$-categories, we already have the work of Balteanu, Fiedorowicz, Schwänzl, and Vogt \cite{MR1982884}.
For general categories, Takeshi Torii developed the more general notion of \emph{duoidal categories} \cite{MR4868047}, and
we expect that $1 \choose 1$-monoidal categories are duoidal categories in which the units coincide. 
It's more or less clear how to adapt Torii's definitions to give a reasonable definition of $m \choose n$-monoidal categories.
A confirmation of our suggested universal properties in that context would amount to a confirmation that our theory is the correct one.

Pushing this story further, can the theory of operator categories \cite{MR3784514} be extended to handle categories like $\DD$ and $\EE$?
A theory of \emph{twofold operator categories}, which could then index all sorts of twofold monoidal structures, would seem particularly interesting.
One supposes that the category $\EE$ should be the terminal twofold operator category.
One can imagine even more general iterated monoidal structures with concomitant theories of iterated operator categories.

We have already seen noncommutative variants of the theory of cographs.
In \cref{sub:1cographs} we defined \emph{$1$-cographs}, the directed version of cographs.
The category $\DD^{[1]}$ of irreflexive $1$-cographs will be the free $\infty \choose 1$-monoidal category generated by a commutative monoid for the left tensor product.
Work of Willow Bevington \& Malthe Sporring seeks to construct various universal $m \choose n$-monoidal categories in similar combinatorial terms.


\section{Factorization algebras}%
\label{sec:factorization}
We arrive at our main construction.
We describe a quite general recipe for defining the category of factorization algebras, and
we give two examples.

Before that, however, we follow that same recipe with one fewer symmetric monoidal structure;
the result is exactly the category of sheaves on the Ran space.

\subsection{Sheaves on the Ran space}%
\label{sub:sheavesonRan}

A symmetric monoidal category is the same thing as a symmetric monoidal functor $\FF \to \Cat$.
Accordingly, we define a \emph{lax symmetric monoidal category}%
\footnote{This is not compatible with other uses of this phrase in the literature.}
to be a lax symmetric monoidal functor $\FF \to \Cat$.
Equivalently, a lax symmetric monoidal category is simply a functor $\FF \to \CAlg(\Cat)$.
These form a bicategory $\CAlg^{\lax}(\Cat)$, which contains $\CAlg(\Cat)$. 

We can endow $\Fun(\FF,\Cat)$ with its Day convolution symmetric monoidal structure $\otimes$.
The unit for this symmetric monoidal structure is the constant diagram at the trivial category $1_{\bullet}$.
Commutative algebras for this symmetric monoidal structure are exactly lax symmetric monoidal categories:
\[
  \CAlg^{\lax}(\Cat) = \CAlg(\Fun(\FF,\Cat)) \period
\]

Every lax symmetric monoidal category $\VV_{\bullet}$ restricts to a functor $\VV_{s,\bullet} \colon \FF_s \to \Cat$.
A \emph{global section} of a lax symmetric monoidal category $\VV_{\bullet}$ is a natural transformation $1_{s,\bullet} \to \VV_{s,\bullet}$.
These form a symmetric monoidal category
\[
  \Gamma(\VV_{\bullet}) \coloneq \HOM_{\Fun(\FF_s,\Cat)}(1_{s,\bullet}, \VV_{s,\bullet}) = \lim_{\angs{n} \in \FF_s} \VV_n \period
\]

Let $\XX$ be a category with finite products, whose objects we think of as in some sense geometric.
Often $\XX$ will be -- or can be enlarged to be -- a topos.
We also imagine some \emph{sheaf theory},
which we give as a lax symmetric monoidal functor $\AA \colon \XX^{\op} \to \Cat$,
or, equivalently, a functor $\AA \colon \XX^{\op} \to \CAlg(\Cat)$.

Consider a functor $X^{\bullet} \colon \FF^{\op} \to \XX$, which is automatically colax symmetric monoidal for the coproduct on $\FF$ and the product on $\XX$.
If $X^{\bullet}$ is actually symmetric monoidal, then it is uniquely specified by $X = X^1$, which is automatically a commutative comonoid.
In any case, we now obtain a lax symmetric monoidal category $\AA(X^{\bullet})$.
Its global sections are now precisely sheaves on the Ran space
\[
  \Gamma(\AA(X^{\bullet})) = \AA(\Ran X) \comma
\]
where $\Ran X = \colim_{\angs{n} \in \FF_s} X^n$
is defined formally as needed.

To obtain sheaves on the \emph{unital} Ran space, one can replace $\FF_s$ in this discussion with the category $\FF_{>0}$ of nonempty finite sets.

\subsection{Parallax symmetric monoidal categories}%
\label{sub:parallax}
Regard $\Cat$ as a twofold symmetric monoidal category $\Cat^{\times, \times}$ with the diagonal twofold structure given by the product.
A symmetric monoidal category is now the same thing as a twofold symmetric monoidal functor $\DD^{\leftsum,\rightsum} \to \Cat^{\times,\times}$.
Accordingly, we define a \emph{parallax symmetric monoidal category} as a lax twofold symmetric monoidal functor $\DD^{\leftsum,\rightsum} \to \Cat^{\times,\times}$.
This is the same thing as a lax symmetric monoidal functor $\DD^{\leftsum} \to \Cat^{\times}$,
or indeed a functor $\DD \to \CAlg(\Cat)$.
These form a bicategory $\CAlg^{\textit{parallax}}(\Cat)$.

Since $\DD$ is a twofold symmetric monoidal category, 
the bicategory $\Fun(\DD,\Cat)$ has its twofold Day convolution symmetric monoidal structure.
The left symmetric monoidal structure is given by the Day convolution
\[
  (\VV \lefttensor \WW)_{\lambda} = \colim_{\angs{\mu \rightsum \nu} \to \angs{\lambda}} \VV_{\mu} \times \WW_{\nu} \comma
\]
and since the disconnected sum $\leftsum$ is the coproduct in $\DD$, the right symmetric monoidal structure is given by the pointwise tensor product
\[
  (\VV \righttensor \WW)_{\lambda} = \VV_{\lambda} \times \WW_{\lambda} \period
\]
The unit in $\Fun(\DD,\Cat)$ is the constant functor $1_{\bullet}$ at the contractible category.

Hence a parallax symmetric monoidal category is precisely the same thing as a commutative monoid for $\righttensor$,
which in turn is also precisely the same thing as a twofold commutative monoid in $\Fun(\DD,\Cat)^{\lefttensor,\righttensor}$:
\[
  \CAlg^{\textit{parallax}}(\Cat) = \CAlg(\Fun(\DD,\Cat)^{\righttensor}) = \CAlg^{(2)}(\Fun(\DD,\Cat)^{\lefttensor,\righttensor}) \period
\]
A parallax symmetric monoidal category $\VV_{\bullet} \colon \DD \to \Cat$ thus comes with products
\[
  \rightboxtensor \colon \VV_{\lambda} \times \VV_{\mu} \to \VV_{\lambda \leftsum \mu} \comma
\]
which in turn induce products
\[
  \leftboxtensor \colon \VV_{\lambda} \times \VV_{\mu} \to \VV_{\lambda \rightsum \mu} \period
\]

In other words, if we regard $\VV_{\bullet}$ as a section of $\Fun(\DD,\Cat)^{\lefttensor,\righttensor} \to \Lambda(\EE)$, then
we are entitled to pull back along the two functors $i,j \colon \Lambda(\FF) \inclusion \Lambda(\EE)$ to extract the two lax symmetric monoidal functors
\[
  \ul{\VV}_{\bullet} \colon \DD^{\rightsum} \to \Cat^{\times} \andeq 
  \ol{\VV}_{\bullet} \colon \DD^{\leftsum} \to \Cat^{\times} \period
\]
Of course $\ol{\VV}_{\bullet}$ contains all the data of $\VV_{\bullet}$,
but the assignment $\VV_{\bullet} \mapsto \ul{\VV}_\bullet$ is a nontrivial forgetful functor
\[
  \CAlg^{\textit{parallax}}(\Cat) \to \CAlg(\Fun(\DD,\Cat)^{\lefttensor}) \period
\]

A parallax symmetric monoidal category $\VV_{\bullet}$ is \emph{presentable} iff
$\VV_{\bullet}$ lifts to a functor $\DD \to \Pr^L$, and $\leftboxtensor$ and $\rightboxtensor$ each preserve colimits separately in each variable.
In this case, $\VV_{\bullet}$ can also be described as a twofold commutative monoid in $\Fun(\DD,\Pr^L)$ with the twofold Day convolution.%
\footnote{Properly speaking, we have not constructed the twofold Day convolution in this case,
because we wanted to use Heine's trick.}

Let $(\XX,\AA)$ be a pair consisting of a category $\XX$ with finite products
along with a sheaf theory $\AA \colon \XX^{\op} \to \Cat$,
as in \cref{sub:sheavesonRan}.
Now let $X^{\bullet}$ be an isolability object of $\XX$.
Then we obtain a parallax symmetric monoidal category
\[
  \AA(X^{\bullet}) \colon \DD \to \Cat \period
\]
The two different symmetric monoidal structures:
\[
  \leftboxtensor \colon \AA(X^\lambda) \times \AA(X^\mu) \to \AA(X^{\lambda \rightsum \mu}) \andeq
  \rightboxtensor \colon \AA(X^\lambda) \times \AA(X^\mu) \to \AA(X^{\lambda \leftsum \mu}) 
\]
are restrictions of the usual external product $\boxtimes \colon \AA(X^{\lambda}) \times \AA(X^{\mu}) \to \AA(X^{\lambda} \times X^{\mu})$.

\subsection{Factorization algebras}%
\label{sub:factorizationalgebras}
Let $\VV_{\bullet} \colon \DD \to \Cat$ be a parallax symmetric monoidal category.
The key to making sense of factorization structures is now to dispense with the lax symmetric monoidality with respect to $\leftsum$ and remember only the lax symmetric monoidality with respect to $\rightsum$.
Thus $\VV_{\bullet}$ restricts to a functor $\ul{\VV}_{s,\bullet} \colon \DD_s \to \Cat$,
which we regard as a nonunital lax symmetric monoidal functor $\DD_s^{\rightsum} \to \Cat^{\times}$ --
or equivalently as a nonunital commutative monoid in $\Fun(\DD_s,\Cat)^{\lefttensor}$.
The assignment $\VV_\bullet \mapsto \ul{\VV}_{s,\bullet}$ is a functor
\[
  \CAlg^{(2)}(\Fun(\DD,\Cat)^{\lefttensor,\righttensor}) \to \CAlg^{\textit{nu}}(\Fun(\DD_s,\Cat)^{\lefttensor}) \period
\]

A \emph{global section} of $\VV_\bullet$ is now a nonunital commutative algebra map
\[
  F_\bullet \colon 1_{s,\bullet} \to \ul{\VV}_{s,\bullet} \period
\]
These form a symmetric monoidal category
\[
  \Gamma(\VV_{\bullet}) = \HOM_{\CAlg^{\textit{nu}}(\Fun(\DD_s,\Cat)^{\lefttensor})}(1_{s,\bullet}, \ul{\VV}_{s,\bullet}) \period
\]

Thus a global section specifies:
\begin{itemize}
  \item for every nonempty cograph $\angs{\lambda}$, an object $F_{\lambda} \in \VV_{\lambda}$,
  \item for every surjection $i \colon \angs{\lambda} \to \angs{\mu}$, an identification $i^{\ast}F_{\lambda} = F_{\mu}$, and
  \item for every pair of cographs $\angs{\lambda}$ and $\angs{\mu}$, an identification $F_{\lambda \rightsum \mu} = F_{\lambda} \leftboxtensor F_{\mu}$ in $\VV_{\lambda \rightsum \mu}$.
\end{itemize}
If $\VV_{\bullet}$ is presented as a symmetric monoidal cocartesian fibration $\VV \to \DD$, then
a factorization algebra is precisely the same thing as a nonunital symmetric monoidal cocartesian section $F_{\bullet} \colon \DD_s \to \VV$.

The global sections of a parallax symmetric monoidal category $\VV_\bullet$ is related to
the global sections of the lax symmetric monoidal category (\cref{sub:sheavesonRan}) given by the restriction $\VV_{\bullet}|\FF$
by a symmetric monoidal functor $\Gamma(\VV_{\bullet}) \to \Gamma(\VV_\bullet|\FF)$.
In general, this functor is not conservative.

Let $(\XX,\AA)$ be a pair consisting of a category $\XX$ with finite products along with a sheaf theory $\AA \colon \XX^{\op} \to \Cat$, as in \cref{sub:parallax}.
Let $X^{\bullet}$ be an isolability object of $\XX$.
Then the global sections of $\AA(X^{\bullet})$ are called \emph{factorization algebras} on $X^{\bullet}$ with coefficients in $\AA$.

\subsection{Factorization algebras in constructible sheaves}%
\label{sub:constructible}

Let $\CC$ be a symmetric monoidal category.
This provides a natural sheaf theory
$X \mapsto \Cnstr(X;\AA) = \Fun(X, \CC)$
on the category of stratified spaces.

For every isolability space $X^{\bullet}$, we have the parallax symmetric monoidal category $\Cnstr(X^{\bullet};\CC)$.
A \emph{constructible factorization algebra} on $X^{\bullet}$ with coefficients in $\CC$ is then a factorization algebra in $\Cnstr(X^{\bullet}; \CC)$.
In the special case that $X^{\bullet}$ is stratified over the canonical isolability poset $K^{\bullet}$ (\cref{sub:isolabilityspaces}), these are called \emph{locally constant factorization algebras}.

Let us characterize the constructible factorization algebras in a fibrational way.
Let $\XX \to \DD$ be the cartesian fibration corresponding to an isolability space $X^{\bullet}$.
We can pull back the cocartesian operad structure on $\XX_s \coloneq \XX \times_{\DD} \DD_s$ along the nonunital operad map $\DD_s^{\rightsum} \to \DD^{\leftsum}$ to obtain an operad:
\[
  \XX_s^{\rightsum} \coloneq \XX^{\sqcup} \times_{\DD^{\leftsum}} \DD_s^{\rightsum} \period
\]
Thus a \enquote{map} $x_1 \rightsum \cdots \rightsum x_n \to y$ in $\XX_s^{\rightsum}$ is a collection of maps $x_i \to y$ of $\XX$ for $1 \leq i \leq n$ that lie over a collection $\angs{\lambda_i} \to \angs{\mu}$ of $\DD$ that fit together into a surjection $\angs{\lambda_1 \rightsum \cdots \rightsum \lambda_n} \to \angs{\mu}$.
Let us call such a \enquote{map} \emph{cartesian} iff which each map $x_i \to y$ is cartesian.

Now a constructible factorization algebra on $X^{\bullet}$ with coefficients in $\CC$ can be described as an operad map $F \colon \XX_s^{\rightsum} \to \CC$ such that for every cartesian \enquote{map} $x_1 \rightsum \cdots \rightsum x_n \to y$,
the induced map
\[
  F(x_1) \otimes \cdots \otimes F(x_n) \to F(y)
\]
is an equivalence.

\subsection{Factorization stacks}%
\label{sub:factorizationstacks}
Let $\XX$ be a topos, and let us consider the pregeometric background $X \mapsto \XX_{/X}$.
This is the most primitive and unstructured setting in which we may consider our constructions.

Let $X^\bullet \colon \DD^{\op} \to \XX$ be an isolability object.
The isolability object $X^{\bullet}$ can be promoted to an isolability topos $\angs{\lambda} \mapsto \XX_{/X^{\lambda}}$ in which all the structure maps are étale morphisms of topoi.
Using the pullback functors in these morphisms of topoi, this defines a parallax symmetric monoidal category.
A factorization algebra in this parallax symmetric monoidal category a \emph{factorization stack} over $X^{\bullet}$.

Let us give a helpful alternative way of describing a factorization stack.
Call a map of nonempty cographs $\angs{\lambda} \to \angs{\mu}$ \emph{attached} iff
every vertex in the image is connected by an edge to every vertex outside the image.
Thus surjective maps of cographs are attached, and
inclusions $\angs{\lambda} \inclusion \angs{\lambda \rightsum \mu}$ are attached.
Let $\DD_{\textit{att}} \subset \DD$ be the subcategory of nonempty irreflexive cographs and attached maps.
An isolability object $X^{\bullet}$ restricts to a functor $X^{\bullet}_{\textit{att}} \colon \DD_{\textit{att}}^{\op} \to \XX$.

A factorization stack over $X^\bullet$ is now a functor $Y^{\bullet} \colon \DD_{\textit{att}}^{\op} \to \XX$ along with a morphism $Y^\bullet \to X^{\bullet}_{\textit{att}}$ such that
\begin{enumerate}
  \item for every surjection $\angs{\lambda} \to \angs{\mu}$, the canonical map
    \[
      Y^{\mu} \to Y^{\lambda} \times_{X^{\lambda}} X^{\mu}
    \]
    is an equivalence, and
  \item for cographs $\angs{\lambda}$ and $\angs{\mu}$, the map
    \[
      Y^{\lambda \rightsum \mu} \to (Y^{\lambda} \times Y^{\mu}) \times_{X^{\lambda} \times X^{\mu}} X^{\lambda \rightsum \mu} 
    \]
    is an equivalence.
\end{enumerate}
(To see that this is an equivalent description, take the target cartesian fibration $\Arr(\XX) \to \XX$, form the dual cocartesian fibration \cite{MR3746613}, and pull this back along $X^{\bullet}$.
Now factorization stacks are symmetric monoidal cocartesian sections of this fibration over $\DD_s$, and this unwinds to the description here.)

In particular, if $X^\bullet$ is additive and $Y^\bullet$ is the restriction of an additive isolability object, then $Y^{\bullet}$ is automatically a factorization stack.
As we shall see in the next subsection, this is not the only way in which factorization stacks arise.

We can now generalize this slightly.
If $U \in \XX$ is an object, then a \emph{factorization groupoid} with object stack $S$
is a functor $Y^{\bullet} \colon \DD_{\textit{att}}^{\op} \to \XX$ along with a morphism $Y^\bullet \to X^{\bullet}_{\textit{att}} \times U \times U$ such that
\begin{enumerate}
  \item for every surjection $\angs{\lambda} \to \angs{\mu}$, the canonical map
    \[
      Y^{\mu} \to Y^{\lambda} \times_{X^{\lambda}} X^{\mu}
    \]
    is an equivalence, and
  \item for cographs $\angs{\lambda}$ and $\angs{\mu}$, the map
    \[
      Y^{\lambda \rightsum \mu} \to (Y^{\lambda} \times_U Y^{\mu}) \times_{X^{\lambda} \times X^{\mu}} X^{\lambda \rightsum \mu} 
    \]
    is an equivalence, where $Y^{\lambda} \to U$ is projection onto the first copy of $U$, and $Y^{\mu} \to U$ is projection onto the second.
\end{enumerate}

\subsection{The Beilinson--Drinfeld Grassmannian in quite a lot of generality}%
\label{sub:heckegrass}

Again let $\XX$ be a topos, and let $X \in \XX$ be an object.
Let $O_X^\bullet \colon \DD^{\op} \to \XX$ be an observer stack with its isolability structure as in \cref{sub:observer}.
As such, $O_X$ is regular (\cref{sub:regularity}) and additive (\cref{sub:additivity}).
Additionally, let us assume that the canonical map $X \to \Obj_X$ factors through $O_X$.

Let $B \in \XX_{/X}$ be any stack over $X$.
We may imagine that $B = BG$ for $G$ a group over $X$, in which case sections of $B \to X$ are $G$-bundles on $X$.
Accordingly, we will call (local or global) sections of $B \to X$ \emph{bundles} on $X$.
These are organized into the moduli stack
\[
  \Bun \coloneq \underline{\Map}_X(X,B) \period
\]
We are going to define a factorization stack over ${\Bun} \times O_X^{\bullet}$ of \emph{modifications} of sections of $B$.

It turns out to be natural to think of moduli stacks of modifications as \enquote{pull-push} expressions.
To explain, let $p \colon U \to X$ and $q \colon U \to Y$.
Attached we have the right adjoint functor
\[
  H_U \coloneq q_{\ast} p^{\ast} \colon \XX_{/X} \to \XX_{/Y} \comma
\]
which carries a stack $B$ over $X$ to the stack over $Y$ given by
\[
  T \mapsto \Map_X(U \times_Y T, B) \period
\]
This is the \emph{pull-push construction} through $U$.

Since pullbacks along étale geometric morphisms satisfy basechange,
the construction $X \mapsto \XX_{/X}$ defines a functor $H \colon \Span(\XX) \to \Pr^{R}$,
where spans are carried to the corresponding pull-push constructions.
In fact, this can be enhanced to a functor of bicategories, and moreover the induced functor
\[
  H \colon (\XX_{/X \times Y})^{\op} \to \Fun^R(\XX_{/X}, \XX_{/Y})
\]
carries colimits to limits.

Now we use the isolability structure on $O_X$ to define a stack $X \odot O_X$ lying over $O_X$ so that,
morally, the fiber over $Z \in O_X$ is $X$ with $Z$ doubled -- \textsc{aka} the \enquote{ravioli space} $X \cup^{X \smallsetminus Z} X$.
The Hecke stack can then be seen as the result of a pull-push construction through $X \odot O_X$.

In fact, we define an isolability stack
\[
  X \odot O_X^{\bullet} \colon \DD^{\op} \to \XX \comma
\]
which carries $\angs{\lambda}$ to
\[
  X \odot O_X^{\lambda} \coloneq X \times_{O_X} (O_X^{1 \leftsum \lambda} \cup^{O_X^{1 \rightsum \lambda}} O_X^{1 \leftsum \lambda}) \period
\]
We have a projection $(p,q^{\lambda}) \colon X \odot O_X^{\lambda} \to X \times O_X^{\lambda}$;
its fibre over a point $(x,Z)$ is the (unreduced) suspension of the object $\enquine{x \cap Z = \varnothing}$.

For a stack $B \in \XX_{/X}$ and an irreflexive cograph $\angs{\lambda}$, we define the \emph{Hecke stack} via the pull-push construction through $X \odot O_X^{\bullet}$:
\[
  \Hecke^{\lambda} = \Hecke^{\lambda}(X,O_X; B) \coloneq H_{X \odot O_X^{\lambda}}(B) = q^{\lambda}_{\ast} p^{\ast} B \in \XX_{/O_X^{\lambda}} \period
\]
Thus a $T$-point of $\Hecke^{\lambda}$ is a $T$-point $Z \in O_X^{\lambda}(T)$, and a map $T \times_{O_X^{\lambda}} (X \odot O_X^{\lambda}) \to B$.
The pushout in the definition of $X \odot O_X^{\lambda}$ gets converted to a pullback:
\[
  \Hecke^{\lambda} = H_{X \times_{O_X} O_X^{1 \leftsum \lambda}}(B) \times_{H_{X \times_{O_X} O_X^{1 \rightsum \lambda}}(B)} H_{X \times O_X^{1 \leftsum \lambda}}(B) \period
\]
Since $O_X^{\bullet}$ is additive, $H_{X \times_{O_X} O_X^{1 \leftsum \lambda}}(B)$ is just the product ${\Bun} \times O_X^{\lambda}$.
This results in the more intuitively appealing formula
\[
  \Map_{O_X^{\lambda}}(T, \Hecke^{\lambda}) = \Map_X(X \times T, B) \times_{\Map_X(X \times_{O_X} O_X^{\ul{1} \rightsum \lambda} \times_{O_X^\lambda} T, B)} \Map_X(X \times T, B) \period
\]
In other words, we have a projection
\[
  \Hecke^{\lambda} \to O_X^\lambda \times {\Bun} \times {\Bun} \comma
\]
whose fiber over a $T$-point $(Z,E_1,E_2)$
-- which is a point $Z \in O_X^{\lambda}(T)$ and a pair $(E_1,E_2)$ of $T$-families of bundles --
is an identification of restrictions to the \enquote{complement} of the graph of $Z$
(formally $X \times_{O_X} O_X^{\ul{1} \rightsum \lambda} \times_{O_X^{\lambda}} T$).
In this sense, the fiber of $\Hecke^{\lambda} \to O_X^{\lambda}$ over $Z$ is the moduli stack of \emph{modifications} of bundles at $Z$. 

The projection $\Hecke^{\lambda} \to O_X^{\lambda}$ is natural in attached maps.
The regularity of $O_X^{\bullet}$ implies that if $\angs{\lambda} \to \angs{\mu}$ is a surjection, then
\[
  \Hecke^{\lambda} \times_{O_X^{\lambda}} O_X^{\mu} = \Hecke^{\mu} \period
\]
Moreover, gluing provides identifications
\[
  \Hecke^{\lambda \rightsum \mu} \equivalence (\Hecke^{\lambda} \times_{\Bun} \Hecke^{\mu}) \times_{O_X^{\lambda \leftsum \mu}} O_X^{\lambda \rightsum \mu} \period
\]
Hence $\Hecke^\bullet$ is a factorization groupoid over $O_X^{\bullet}$ with object-stack $\Bun$.

If the stack $B$ comes with a global basepoint $P \in \Map_X(X,B)$, then
the fiber of $\Hecke^{\bullet} \to \Bun$ over $P$ is the \emph{Beilinson--Drinfeld Grassmannian} $\Grass_\bullet(X, O_X; B, P)$ for $B$ on $O_X^\bullet$ relative to $P$.

If we work in the context of \cref{sub:FFcurve}, then one only gets the correct object by passing to smaller versions of these stacks satisfying meromorphy conditions.

\subsection{Comments \& questions}%
\label{sub:comments-factorization}

Let $\XX$ be a category with finite products, and let $\AA$ be a sheaf of symmetric monoidal categories on $\XX$.
Let $X$ be an object of $\XX$.
As we discovered (\ref{sub:sheavesonRan}), laxifying the symmetric monoidal structure on $\AA(X)$ by using the powers $X^n$ gives a lax symmetric monoidal functor $\FF_s \to \Cat$.
The non-multiplicative global sections of this functor are then precisely sheaves on the Ran space.

The main construction here laxifies the symmetric monoidal structure further to a parallax symmetric monoidal structure by using an isolability structure on $X$.
After dispensing, with the unit, we are left with a lax twofold symmetric monoidal functor $\DD_s \to \Cat$.
Factorization algebras are then global sections of this functor that are multiplicative with respect to the \emph{second} symmetric monoidal structure only.

This now suggests a strategy for building a theory of factorization algebras with \enquote{extended observables} of different dimensions:
\begin{enumerate}
  \item We consider $(1+n)$-fold symmetric monoidal categories, which have tensor products
    \[
      \otimes_0 \to \otimes_1 \to \cdots \to \otimes_n \period
    \]
  \item The free $(1+n)$-fold symmetric monoidal category generated by a commutative monoid for the leftmost tensor product $\otimes_0$ is a certain category $\DD(n)$ of \emph{$n$-cographs},
    which are graphs with a coloring of each edge by one of $n$ colors satisfying conditions analogous to those we saw in $\DD = \DD(1)$.
  \item An \emph{$n$-isolability object} is then a functor $X^{\bullet} \colon \DD(n)^{\op} \to \XX$.
  \item Applying $\AA$, we then obtain an \emph{$n$-parallax symmetric monoidal category} $\AA(X^{\bullet})$, which is a $(1+n)$-fold lax symmetric monoidal functor $\DD(n) \to \Cat$.
    These are a kind of laxification of symmetric monoidal categories.
  \item Now \emph{$n$-dimensional factorization algebras} should be nonunital $n$-fold symmetric monoidal functors $1_{\bullet} \to \AA(X^{\bullet})$, relative only to the tensor products $\otimes_1, \dots, \otimes_n$.
    Of course other variants are now possible as well.
\end{enumerate}

This is, of course, speculative.
We do not yet know how to construct physically meaningful examples of $n$-isolability objects.


\DeclareFieldFormat{labelnumberwidth}{#1}
\printbibliography[keyword=alph, heading=references]
\DeclareFieldFormat{labelnumberwidth}{{#1\adddot\midsentence}}
\printbibliography[heading=none, notkeyword=alph]

\end{document}